\documentclass[reqno]{amsart}

\title{On the Bauer--Furuta construction}
\author{Takumi Maegawa}
\address{Graduate School of Mathematical Science, the University of Tokyo}
\email{\url{tmaegawa@fuji.waseda.jp}}

\usepackage{geometry}
\usepackage{amsmath, amssymb, amsthm, mathtools, tikz-cd}
\numberwithin{equation}{section}
\usepackage{url, xcolor, hyperref, footnotehyper}
\definecolor{darkgreen}{rgb}{0,0.30,0} 
\definecolor{darkred}{rgb}{0.75,0,0}
\definecolor{darkblue}{rgb}{0,0,0.6}
\definecolor{lightblue}{rgb}{0,0.5,0.73}
\makeatletter
\def\@footnotecolor{red}
\define@key{Hyp}{footnotecolor}{%
 \HyColor@HyperrefColor{#1}\@footnotecolor%
}
\def\@footnotemark{%
  \leavevmode
  \ifhmode\edef\@x@sf{\the\spacefactor}\nobreak\fi
  \stepcounter{Hfootnote}%
  \global\let\Hy@saved@currentHref\@currentHref
  \hyper@makecurrent{Hfootnote}%
  \global\let\Hy@footnote@currentHref\@currentHref
  \global\let\@currentHref\Hy@saved@currentHref
  \hyper@linkstart{footnote}{\Hy@footnote@currentHref}%
  \@makefnmark
  \hyper@linkend
  \ifhmode\spacefactor\@x@sf\fi
  \relax
}
\makeatother
\hypersetup{
    colorlinks=true,
    footnotecolor=darkred,
    citecolor=lightblue,
    linkcolor=magenta,
    urlcolor=darkblue
}

\usepackage[capitalize]{cleveref}
\usepackage{comment}

\usepackage[T1]{fontenc}
\usepackage{lmodern}
\usepackage[utf8]{inputenc}
\usepackage[cal=cm, scr=euler, scrscaled=1.05]{mathalpha}
\usepackage{inconsolata}

\theoremstyle{plain}
\newtheorem{theorem}{Theorem}[subsection]
\newtheorem{lemma}[theorem]{Lemma}
\newtheorem{corollary}[theorem]{Corollary}
\newtheorem{proposition}[theorem]{Proposition}
\theoremstyle{definition}
\newtheorem{definition}[theorem]{Definition}
\newtheorem{example}[theorem]{Example}
\newtheorem{remark}[theorem]{Remark}
\newtheorem{convention}[theorem]{Convention}
\newtheorem{construction}[theorem]{Construction}
\newtheorem{notation}[theorem]{Notation}
\newtheorem{summary}{Summary}[section]
\newtheorem{remark*}[summary]{Remark}
\crefname{convention}{Convention}{Conventions}
\crefname{construction}{Construction}{Constructions}

\newcommand{\Sp}{\mathsf{Sp}}  % spectra
\newcommand{\an}{\mathsf{An}} % anima
\newcommand{\cat}{\mathsf{Cat}} % small oo-cats
\newcommand{\Cat}{\widehat{\cat}} % locally small oo-categories
\newcommand{\Grp}{\mathsf{Grp}} % groups
\newcommand{\CAlg}{\mathsf{CAlg}} % commutative algebras
\newcommand{\Mod}{\mathsf{Mod}} % bimodules
\renewcommand{\Pr}{\mathsf{Pr}} % presentable categories
\DeclareMathOperator{\Pro}{\mathsf{Pro}} %\newcommand{\Pro}{\mathsf{Pro}} % pro-objects
\newcommand{\Sh}{\mathsf{Sh}} % sheaves
\newcommand{\PSh}{\mathscr{P}} % presheaves/free cocompletion
\newcommand{\DDelta}{\boldsymbol{\Delta}} % the simplex category
\newcommand{\Top}{\mathsf{Top}} % topological spaces / topoi
\newcommand{\Locale}{\mathsf{Locale}} % locale
\newcommand{\Op}{\mathrm{Open}} % Open poset
\newcommand{\Sm}{\mathrm{Sm}} % smooth spaces
\newcommand{\SH}{\mathsf{SH}} % The stable homotopy category
\newcommand{\Corr}{\mathsf{Corr}} % correspondences
\newcommand{\Mfd}{\mathsf{Mfld}} % manifolds
\newcommand{\St}{\mathsf{St}} % stacks
\newcommand{\Lie}{\mathsf{Lie}} % Lie groups
\newcommand{\LocSys}{\mathrm{LocSys}} % local systems
\newcommand{\Pic}{\mathsf{Pic}} % the Picard groupoid
\newcommand{\Ban}{\mathsf{Ban}} % Banach
\newcommand{\Glo}{\mathsf{Glo}} % global

\DeclareMathOperator*{\colim}{colim} % colimit
\let \lim \relax
\DeclareMathOperator*{\lim}{lim} % limit
\DeclareMathOperator{\Fun}{\mathsf{Fun}} % functor category
\DeclareMathOperator{\Aut}{Aut} % automorphism group
\DeclareMathOperator{\Iso}{Iso} % isomorphisms
\DeclareMathOperator{\Hom}{Hom} % hom
\let \hom \relax
\DeclareMathOperator{\hom}{\underline{Hom}} % internal hom
\DeclareMathOperator{\im}{im} % image
\DeclareMathOperator{\cof}{cof} % cofiber
\let \ker \relax
\DeclareMathOperator{\ker}{ker} % kernel
\DeclareMathOperator{\cok}{cok} % cokernel
\DeclareMathOperator{\Gal}{Gal} % Galois group
\DeclareMathOperator{\GL}{GL} % general linear
\renewcommand{\H}{\mathrm{H}} % cohomology
\newcommand{\Sing}{\mathrm{Sing}} % singular
\newcommand{\oo}{\infty} % infinity
\newcommand{\pt}{\mathrm{pt}} % point
\newcommand{\op}{\mathrm{op}} % opposite
\newcommand{\id}{\mathrm{id}} % identity
\newcommand{\sslash}{{\mathbin{/\mkern-5mu/}}} % clever quotient
\newcommand{\pr}{\mathrm{pr}} % projection
\newcommand{\unit}{\boldsymbol{1}} % monoidal unit 1
\newcommand{\diag}{\mathrm{diag}} % diagonal
\newcommand{\ind}{\mathrm{ind}} % induction
\newcommand{\st}{\mathrm{st}} % stable
\newcommand{\lex}{\mathrm{lex}} % left exact
\renewcommand{\L}{\mathrm{L}} % left adjoint
\newcommand{\R}{\mathrm{R}} % right adjoint
\renewcommand{\S}{\mathbb{S}} % sphere spectrum
\newcommand{\Z}{\mathbb{Z}} % integers
\newcommand{\RR}{\mathbb{R}} % real numbers
\newcommand{\CC}{\mathbb{C}} % complex numbers
\newcommand{\HH}{\mathbb{H}} % quaternions
\renewcommand{\AA}{\mathbb{A}} % affine
\newcommand{\Th}{\mathrm{Th}} % Thom sheaf
\newcommand{\intr}{\mathrm{int}} % interior

\DeclareMathOperator{\Spin}{Spin} % spin group
\DeclareMathOperator{\Pin}{Pin} % pin group
\DeclareMathOperator{\SO}{SO} % special orthogonal
\DeclareMathOperator{\U}{\mathit{U}} % unitary group
\DeclareMathOperator{\USp}{Sp} % compact symplectic group
\newcommand{\Spec}{\mathrm{Spec}} % Spec(R)
\newcommand{\Vect}{\mathrm{Vect}} % vector bundles
\newcommand{\Rep}{\mathsf{Rep}} % Representation category
\newcommand{\BM}{\mathrm{BM}} % Borel--Moore
\newcommand{\cnst}{\mathrm{cnst}} % constant
\newcommand{\lcnst}{\,\mathrm{l.c.}} % locally constant
\newcommand{\dR}{\mathrm{dR}} % de Rhamm
\newcommand{\topl}{\mathrm{top}} % topological
\newcommand{\BF}{\mathrm{BF}} % Bauer--Furuta
\newcommand{\Fred}{\mathrm{Fred}} % Fredholm

\makeatletter
\renewenvironment{proof}[1][\relax]{\par
  \pushQED{\qed}%
  \normalfont \topsep6\p@\@plus6\p@\relax
  \trivlist
  \item[\hskip\labelsep\itshape
    \ifx#1\relax \proofname\else\proofname{} of #1\fi\@addpunct{.}]\ignorespaces
}{%
  \popQED\endtrivlist\@endpefalse
}
\makeatother
\begin{document}

\begin{abstract}
  Using the six-functor formalism for sheaves of spectra on topological spaces,
  we provide a novel construction of the Bauer--Furuta invariant, as well as its family version.
  This approach avoids the conventional arguments based on approximations by
  finite-dimensional subspaces, and we instead employ the Borel--Moore homology spectra relative to Fredholm maps between Banach spaces.
  A key observation here is that \( C^1 \)-differentiable Fredholm maps between Banach manifolds
  are locally proper, thereby defining the shriek functors,
  whose dualizing objects may be described as the Thom
  spectra of the Atiyah--Singer families index.

  We also outline a possible candidate for the stable homotopy theory of genuine
  equivariant sheaves on topological spaces with Lie group actions.
  In this context, we investigate the proper pushforward functor, which
  accommodates the genuine equivariant Bauer--Furuta invariant.
\end{abstract}
\maketitle

\setcounter{tocdepth}{1}
\tableofcontents

\section{Introduction}

Let \( (X, g, \mathfrak{s}) \) be a closed riemannian
spin\( ^c \) \( 4 \)-manifold equipped with a connection
on the determinant line bundle.
In the proof~\cite{Fur01} of the weaker version of
\emph{the \(\frac{11}{8}\)-conjecture}, Furuta used the
Seiberg--Witten equation to obtain a \(U(1)\)-equivariant%
\footnote{If \(\mathfrak{s}\) comes from a spin structure, then it lifts to a \(\mathrm{Pin}(2)\)-equivariant map.}
proper map between \(\mathbb{R}\)-Hilbert spaces%
\footnote{\(L^2_k\) is the Sobolev completions, \(k>4\), \(T^\ast\) is the cotangent bundle of \(X\), \(S^\pm\) is the positive/negative spinor bundles, and \(\Lambda^{2,+}\) is the \(({+}1)\)-eigenspace for the Hodge star operator on \(\Lambda^2 T^\ast\).}
%%%%%%%%%
\begin{equation}\label{eq:SWmap}
    f \colon L^2_k (X; T^\ast \oplus S^+) \times H^0(X;\mathbb{R}) \longrightarrow L^2_{k-1} (X; S^- \oplus \Lambda^{2,+} T^\ast \oplus \Lambda^0 T^\ast)
\end{equation}
and construct a \(U(1)\)-equivariant map
\[S^{2\ind_{\mathbb{C}} \not{D} + b_1(X)} \to S^{b_2^+(X)}\]
between (stable) spheres.
It can be better described as a map
\[\BF_f \colon \S^{2\ind_{\mathbb{C}} \not{D}} \to \S^{b_2^+(X)}\]
between (genuine) equivariant sphere spectra over the base space
\(H^1(X;\mathbb{R})/H^1(X;\Z)\) as defined in~\cite{BF1}.
The resulting map between sphere spectra was shown to be independent of the choices of metrics and connections.
This is referred to as the \emph{Bauer--Furuta invariant} of
\((X,\mathfrak{s})\).

The Bauer--Furuta invariant, as previously mentioned, has been utilized in the partial progress toward the \( \frac{11}{8} \)-conjecture, a geography problem for closed 4-manifolds, which is most extensively studied in~\cite{HLSX18}. 
More recently, a family version of the Bauer--Furuta invariant has been employed to prove the existence of exotic diffeomorphisms as in~\cite{BaragliaKonno19} and subsequent works. However, the Bauer--Furuta invariants for families have yet to be studied in full generality.
It is also worth noting that the Bauer--Furuta invariant should be part of a broader framework that resembles a topological quantum field theory.
To ensure the \((\infty, 0)\)-functoriality of such a topological field theory, at least a comprehensive understanding of the family version of the Bauer--Furuta invariant is required.

%The Bauer--Furuta map \(\BF_f\) is usually constructed as follows.
Let us briefly recall the construction of the Bauer--Furuta invariant.
The argument presented below is referred to as the
\emph{finite-dimensional approximations}.
Let \(f\colon \mathcal{H}' \to \mathcal{H}\) simply denote the Seiberg--Witten
map \eqref{eq:SWmap}.
Consider a (sufficiently large) finite-dimensional (equivariant) subspace
\(W \subset \mathcal{H}\).
The linear part \(l\coloneqq df\) of the Seiberg--Witten map \(f\) has the
property of being Fredholm, and defines a finite-dimensional subspace
\(W'=l^{-1}(W)\) of \(\mathcal{H}'\). Using the orthogonal projection
\(\pr_W \colon \mathcal{H} \to W\), one can show that%
\footnote{See~\cite[Lemma 3.4]{Fur01} or~\cite[Lemma 2.3]{BF1}.}
there exists a bounded open neighborhood \(N\subset \mathcal{H}'\) of the zero
set \(f^{-1}(0)\) such that the composite
\begin{equation}\label{eq:fin-dimApprox}
  \begin{tikzcd}
    f_W^\mathrm{apprx} \colon W' \cap \overline{N} \ar[r, hookrightarrow]
    & \mathcal{H}' \ar[rr, "f"]
    && \mathcal{H} \ar[rr, "\pr_W"]
    && W
  \end{tikzcd}
\end{equation}
maps \(W' \cap (\overline{N}-N)\) to \(W-\{0\}\).%
\footnote{\(f_W^\mathrm{apprx}\) itself may not be a proper map.}
Thus, via the Pontrjagin--Thom collapse, we obtain a map on spheres 
\begin{equation}\label{eq:BFapprx}
  \begin{tikzcd}[column sep=large]
    S^{W'} \ar[r, "\textrm{collapse}"] & \displaystyle
    \frac{W' \cap \overline{N}}{W' \cap (\overline{N}-N)} \ar[r, "f_W^\mathrm{apprx}"] & \displaystyle \frac{W}{W-\{0\}} \simeq S^W.
  \end{tikzcd}
\end{equation}
Taking \(\displaystyle\colim_{W\to \oo} \Sigma^{\oo-W}\)%
\footnote{Although the functoriality in \( W \) is not very clear at this point. These colimits actually stabilize at some \(W\).}
on both sides, we get the desired map
\begin{equation}\label{eq:BFintro}
  \BF_f\colon \S^{\ind{(df)}} \to \S.
\end{equation}

This construction involves certain inconveniences,
as some desirable properties, such as
functoriality or independence of the choices of
metrics/connections/perturbations, are not so obvious.
Also, generalizing this construction to families over an arbitrary base
may be nontrivial%
\footnote{Especially when a base is noncompact nor non-CW.}
and will require some work.

The purpose of this paper is to present a solid argument for defining the Bauer--Furuta invariants in a completely canonical way.
This new approach is based on the theory of six-functor formalisms and does not rely on the original argument of the Bauer--Furuta construction.
Our starting point is as follows.
We propose, among other things, that, we should avoid using 
\(\S^{\ind(df)}\) as it implicitly incorporates some coordinates on the sphere.
Instead, we suggest using the ``coordinate-free'' form, namely \(f^!(\S)\),
which offers better universality and functoriality.
We summarize the situation as follows.
\begin{summary}
    Consider the topological space of linear Fredholm operators
    \(\Fred(\mathcal{H}',\mathcal{H})\) between \(\mathbb{R}\)-Banach spaces,
    endowed with the norm topology.
    We have the family of linear Fredholm maps as follows.
    \[
      \begin{tikzcd}[row sep=small]
        \Fred(\mathcal{H}',\mathcal{H}) \times \mathcal{H}' \ar[rd, "p"'] \ar[rr, "a"] && \Fred(\mathcal{H}',\mathcal{H}) \times \mathcal{H} \ar[ld] \\
        & \Fred(\mathcal{H}',\mathcal{H}) &
      \end{tikzcd}
    \] Then:
    \begin{enumerate}
        \item The map \(a\) is locally proper. Thus, the functor \(a^!\) is well-defined for sheaves of spectra.
        \item The resulting sheaf \[a^!(\unit) \in \Sh\left(\Fred(\mathcal{H}',\mathcal{H}){\times} \mathcal{H}'\,;\,\Sp\right)\] is locally constant, also constant along \(p\), and is of the form \(\Th(\ind(a))\), the Thom spectrum sheaf for the (families) index \(\ind(a)\), which lives in the \(ko\)-sheaf-cohomology.
        \item In fact, the resulting monodromy local system \[
        p_\ast a^!(\unit) \colon \Pi_\oo(\Fred(\mathcal{H}',\mathcal{H})) \to \Sp
        \] can model the J-homomorphism functor \(J\colon \Omega^\oo ko \to \Sp.\)
        \item Moreover, the \emph{linearization hypothesis} holds in this context: Let \(f\colon \mathcal{H}' \to \mathcal{H}\) be a \(C^1\)-differentiable map which has Fredholm differentials. Then \(f\) is also locally proper, and we have an identification \[
        f^!(\unit) \simeq (df)^!(\unit) \simeq \Th(\ind(df)).
        \]
    \end{enumerate}
    These observations will perfectly elucidate the appearance of
    \(\S^{\ind(df)}\) in %the Bauer--Furuta map
    \eqref{eq:BFintro}, and the
    Bauer--Furuta map itself can be given as the usual proper pushforward map
    on the Borel--Moore homology spectra.
\end{summary}

Since we aim to handle those infinite-dimensional manifolds,
it is crucial to extend the six-functor formalism which was originally
developed for locally compact Hausdorff spaces to include locally proper maps
(between possibly locally non-compact spaces).
It is also remarkable that most attempts to construct homotopy theoretic
invariants for families have opted to use the \emph{local systems} of spectra
rather than sheaves of spectra.
For our purpose, we find that the sheaf-based approach is significantly more
natural and effective than local systems.
This is particularly evident in cases involving locally proper maps between
Banach spaces, where the lack of geometrically convenient compactifications
makes understanding Borel--Moore homology through local systems highly
impractical.

We would like to warn the reader that this paper mainly focuses on the
\emph{non-equivariant} version of the Bauer--Furuta construction, which is
the content of \cref{section:BF}.
The group equivariance, however, plays a crucial role in the applications of
the Bauer--Furuta invariants to low-dimensional topology.
To obtain such a genuine equivariant refinement in a parallel way, we need
some genuine equivariant six-functor formalism, which we hope to exist but is
not fully developed yet.
In \cref{section:SHG}, we will outline a portion of such a genuine equivariant
six-functor formalism, and we at least define the genuine equivariant version
of the Bauer--Furuta construction.
Detailed studies on the properties of the genuine equivariant Bauer--Furuta
construction, as defined so, will be explored in a subsequent paper.

\begin{remark*}
  By its very construction,
  the Seiberg--Witten map, considered as a map between Banach vector spaces, can be roughly seen as a polynomial map (in infinitely many variables).
  Thus,
  one might hope that the Bauer--Furuta map takes the form of the Betti realization of a map in some motivic stable category such as \( \SH(\Spec(\RR)/\Pin(2)) \). 
  Such a motivic refinement of the Bauer--Furuta invariant is expected
  to resolve the known nilpotence phenomena of the Bauer--Furuta invariant, which is one of the obstructions
  to provide the \( \frac{11}{8} \)-inequality from the Bauer--Furuta invariant.

  Our reformulation of the Bauer--Furuta invariant may also be of interest in this direction.
  In fact, %observe that
  the original Bauer--Furuta construction, as we presented in \eqref{eq:BFapprx}, involves the Pontrjagin--Thom collapsing map, which exists by excision in algebraic topology.
  The motivic stable homotopy theory, on the other hand, supports excision for Zariski open subspaces and not for arbitrary open subspace of the Betti realization of Affine planes.
  For this reason, it was hard to imagine the existence of a motivic lift of the original Bauer--Furuta construction.
  In this paper, however, we will see that the Bauer--Furuta map can be defined solely in terms of the proper pushforward.
  This approach bypasses the need for excision theorems that are specific to topology.
\end{remark*}

\subsection*{Acknowledgments}
The author would like to thank Lars Hesselholt for fruitful discussions and advice, especially on the use of six-functor formalism.
%The author would like to thank Yosuke Morita for introducing the author to this problem and for helpful discussions.
%
The author would also like to thank Yosuke Morita, Ko Aoki, Bastiaan Cnossen, Mikio Furuta, Ryomei Iwasa, Hokuto Konno, Jin Miyazawa and Zhouli Xu for various comments and valuable information.
In particular, the author is grateful to Bastiaan Cnossen and Marco Volpe for helpful discussions about the potential of genuine equivariant six-functors and for sharing details about their ongoing research projects.
This project was supported by JSPS KAKENHI Grant Number 24KJ0795 and the WINGS-FMSP program at the Graduate School of Mathematical Science, the University of Tokyo.

\subsection{Notations}
\emph{Categories} mean \((\oo,1)\)-categories.
\(\cat\), \(\an\), \(\Pr^\L\), \(\Cat\) and \(\Sp\) denote the category 
of small categories, the category 
of small \((\oo,0)\)-categories, the category
of presentable categories (with morphisms the left adjoint functors), the category 
of locally small categories, and the category 
of spectra, respectively. 
\(\Pr^\L\) is equipped with the symmetric monoidal structure which gives
the tensor product of presentable categories. 
In a \emph{presentably symmetric monoidal category}, i.e.~an object of \(\CAlg(\Pr^\L)\), we use \(\otimes\),
\(\unit\), and \(\hom\) to denote the tensor product, the
monoidal unit and the internal hom, respectively.
For example, \(\Sp\) is equipped with the presentably symmetric monoidal
structure, which is an idempotent algebra in \(\Pr^\L\) corresponding
to the localization onto the full subcategory \(\Pr^\L_\st\) of
\emph{stable presentable categories}. The monoidal unit for \(\Sp\), the
\emph{sphere spectrum}, is also denoted by \(\S\).

\section{The (non-equivariant) Bauer--Furuta map} \label{section:BF}

\begin{convention}\label{conv:allHaus}
    In this section, all topological spaces are assumed to be Hausdorff
    for simplicity%
    \footnote{Another option is to assume that every map we will care is separated.},
    so that universally closed maps are proper and in particular locally
    proper 
    %Our convention on topological space terminology can be found in
    (\cref{conv:(loc)proper}).

    Sheaves will always take values in \(\Sp\), the category of spectra.
    For a topological space \(Y\), \(\Sh(Y)\) will denote the category of
    sheaves of spectra \(\Sh(Y;\Sp)\) defined in~\cite[6.2.2.6, 6.3.5.15]{HTT},
    see \cref{notation:Sh(Y)} also.
    We use the six operations \(f^\ast, f_\ast, f_!, f^!, \otimes, \hom\) on
    those sheaves, which are summarized in \cref{section:top6FF}
    (especially in \cref{thm:top6FF}). 
    A base topological space will often be denoted by \(S\).
\end{convention}

\subsection{From Borel--Moore functoriality to Bauer--Furuta}

\begin{construction}\label{const:BF}
  Let \(f\colon L \to Y\) be a proper map between topological spaces over
  a base \(S\).
  \begin{equation}\label{eq:f/S}
      \begin{tikzcd}[row sep = small]
          L \ar[rr, "f"] \ar[rd, "r"'] & & Y \ar[dl, "q"] \\
          & S &
      \end{tikzcd}
  \end{equation}
  Since \(f\) is proper, we have the functor \(f^!\) right adjoint to \(f_\ast\).
  Thus, we have the map (in \(\Sh(S)\)) of the following form.
  \[
    \BF_f \colon r_\ast f^!(\unit) = q_\ast f_\ast f^!(\unit) \xrightarrow{q_\ast(\mathrm{counit})} q_\ast(\unit)
  \]
  {This will be our definition of the \emph{Bauer--Furuta map}.}
\end{construction}

It is nothing but the usual proper pushforward map
\(\H^\BM_\bullet(L/Y;\S) \to \H^\BM_\bullet(Y/Y;\S) = \unit_{\Sh(Y)}\)
on the Borel--Moore homology, applied to the functor
\(q_\ast\colon \Sh(Y) \to \Sh(S)\).
An immediate observation is that this construction is pullback-stable:
\begin{remark}
  \label{rem:BFisPBstable}
  Consider the following diagram of topological spaces in which \(f\) is proper.
  \[
    \begin{tikzcd}
      L' \ar[r, "f'"] \ar[d, "g'"'] \ar[dr, phantom, very near start, "\lrcorner"]	& Y' \ar[d, "g"] \ar[dr, "q'"] \\
      L \ar[r, "f"']	& Y \ar[r, "q"'] & S
    \end{tikzcd}
  \]
  Then we have a commutative square of the following form.
  \[
    \begin{tikzcd}[row sep = small]
      f_\ast f^! \ar[r, "\mathrm{counit}"] \ar[d, "f_\ast f^!(\mathrm{unit})"'] & \id \ar[d, "\mathrm{unit}"] \\
      f_\ast f^! g_\ast g^\ast \ar[r] & g_\ast g^\ast
    \end{tikzcd}
  \]
  Here the bottom horizontal map is the composite of the proper basechange
  isomorphism and the counit map as follows.
  \[\begin{tikzcd}[column sep = large]
    f_\ast f^! g_\ast g^\ast & \ar[l, "\sim"', "\mathrm{}"] f_\ast g'_\ast f'^! g^\ast = g_\ast f'_\ast f'^! g^\ast \ar[r, "g_\ast(\mathrm{counit})"] & g_\ast g^\ast
  \end{tikzcd}\]
  Here the first map is the isomorphism obtained by taking the right adjoints
  of the (proper) basechange isomorphism
  \( g^\ast f_\ast \simeq f'_\ast g'^\ast \).
  In particular, we have a diagram
  \begin{equation}\label{eq:g^star(BF)}
    \begin{tikzcd}[row sep = small]
      \mathllap{\BF_f \colon} q_\ast f_\ast f^!(\unit) \ar[r] \ar[d] & q_\ast \unit \ar[d] \\
      \mathllap{\BF_{f'} \colon} q'_\ast f'_\ast f'^!(\unit) \ar[r] & q'_\ast \unit.
    \end{tikzcd}
  \end{equation}
  Note that if \(g\) is contractible (\cref{dfn:contractible}), then both
  vertical maps are isomorphisms,
  so that we identify
  \(\BF_f\) with \(\BF_{f'}\).

  The argument here is functorial at least up to homotopy:
  If we further basechange along \(h\colon Y'' \to Y'\), the commutative
  squares given above compose and form a rectangle.
\end{remark}

The rest of this section is devoted to explaining how this map \( \BF_f \)
subsumes the construction of~\cite{BF1}.
We here record the abstraction of basic properties satisfied by the
Seiberg--Witten map \eqref{eq:SWmap}.

\begin{remark}\label{notation:abstractSWmap}
    We will be interested in the following situation:
    \begin{equation*}
        \begin{tikzcd}[row sep=small]
            L \ar[rr, "f"] \ar[rd, "r"'] & & Y \ar[dl, "q"] \\
            & S &
        \end{tikzcd}
    \end{equation*}
    \(f\) is a \(C^1\)-differentiable (\cref{dfn:bundleC^1}) and proper map
    between Banach vector bundles (\cref{conv:HilbBdl}) over \(S\) such that
    the (vertical) differential \(df \colon L \times_S L \to L \times_S Y\) is a linear Fredholm map at every point. 
    In practice, it is often the case that any differentials at two distinct
    points only differ by a compact linear operator.
\end{remark}

\begin{example}
    [{\cite{BF1}}] \label{eg:familySWmap}
    Let 
        \(\begin{tikzcd}[row sep=small, column sep=small]
            E' \ar[rr, "f"] \ar[rd, "r"'] & & E \ar[dl, "q"] \\
            & S &
        \end{tikzcd}\)
    be a map between Hilbert vector bundles which is given as \(f=l+c\)
    for \(l\) a linear Fredholm map and \(c\) a compact map
    (in the sense that \(c\)  maps a bounded disk bundle of \(E'_K\) to a
    relatively compact subset in \(E_K\) for each compact subset \(K \subset S\)).
    Assume that on each fiber, the map \(f_x \colon E'_x \to E_x\) has bounded
    preimages of bounded sets.%
    \footnote{This is the case for the actual Seiberg--Witten map \eqref{eq:SWmap}.}
    Then~{\cite[Lemma 2.2]{BF1}} says that each map \(f_x\colon E'_x \to E_x\)
    on fibers is proper.
    In the same way, it follows that if \(S\) is compactly generated, then
    \(f\colon E' \to E\) is a proper map.
    See \cref{lem:familyProper} for a discussion.
\end{example}

\begin{remark}
    A typical example of a base \(S\) that will be considered in the
    Bauer--Furuta construction is either the torus
    \(H^1(X;\RR)/H^1(X;\Z)\), the space of riemannian metrics and connections,
    or the space of spin\(^c\)-structures.
    We will remark on this point in \cref{section:SW}. In particular, we will
    need a locally non-compact base space \(S\).
\end{remark}

We have a fundamental decomposition for Fredholm maps as follows.

\begin{lemma}\label{lem:factorizaton}
    Let \(f\colon L \to Y\) be a \(C^1\)-differentiable map between Banach
    vector bundles over \(S\) with Fredholm differentials.
    Then locally on the source, the map \(f\) factors as a composite of the
    following form.
    \begin{equation*}
        \begin{tikzcd}
            L\; \ar[rrrd, "f"', start anchor=south] \ar[r, "\supset", phantom] & U \ar[rr, "e"] & & U \times_Y T \ar[d, "p"] \\
            &&& Y
        \end{tikzcd}
    \end{equation*}
    Here \(T\) is a finite-rank vector bundle over an open subspace of \(Y\),
    \(e\) is the zero-section, and \(p\) is cohomologically smooth.
\end{lemma}
\begin{proof}
    Given a point \(v\in L_x \subset L\) on the source, let \(T_0\) denote a closed linear subspace of \(Y_x\) complementary to \(\im(d_vf)\), which exists since the derivative is Fredholm.
    Choose a neighborhood \(S_0\) of \(x\) and trivializations \(L\mid_{S_0} \cong S_0 \times L_x\) and \(Y\mid_{S_0} \cong S_0\times Y_x\).
    Define the map \(p\colon S_0 \times L_x \times T_0 \to S_0 \times Y_x\) by the formula
    \[
    p(x', v', w) = (x', f_{}(v') + w),
    \]
    so that it is of \(C^1\)-class and the (vertical) derivative at the point \((x,v,0)\) is surjective and has kernel isomorphic to the kernel of \(d_vf\). 
    Therefore, by virtue of \cref{cor:submersion}, we find a vector space \(K_0\) isomorphic to \(\ker(d_v f)\) and open neighborhoods \(V\) of \(f(v)\) in \(Y\), \(E\) of \(0\) in \(T_0\) and \(U\) of \(v\) in \(L\) such that the map \(p\colon S_0 \times_S L \to Y\) is locally of the following form. \[
    \begin{tikzcd}
        U \times E \ar[r, "\cong", phantom] & V \times K_0 \ar[r, "\pr"] & V
    \end{tikzcd}
    \] Since \(K_0\) is finite-dimensional, it follows that \(p\) is cohomologically smooth (\cref{prop:PD}).
    The desired factorization is obvious from the definition of \(p\).
\end{proof}

We have the following corollaries.

\begin{corollary}
    \label{cor:FredLocProper}
    Let \(f\) be as in \cref{lem:factorizaton}.
    Then
    \(f\) is locally proper (\cref{conv:(loc)proper}). In particular, the functor \(f^!\colon \Sh(Y) \to \Sh(L)\) is well-defined (\cref{thm:top6FF}).
\end{corollary}
\begin{proof}
    The property of being locally proper can be checked locally on the source and the target.
    By \cref{lem:factorizaton},
    \(f\) is locally a composition of \(e\), which is proper, and \(p\), which is locally proper.
\end{proof}

Before stating the next corollary, let us recall the construction of Thom spectra.

\begin{construction}\label{const:Th(ind(l))}
  For a finite-rank vector bundle \( q\colon V \to S \), define the \emph{Thom spectrum sheaf} \( \Th(V) \) to be \( q_\ast q^!(\unit)\in \Sh(S) \).
  It is also isomorphic to \( q_\sharp \cof(j_!(\unit)\to \unit) \) where \( j\colon V^\times \to V \) is the complement of the zero-section, which is informally interpreted as the homotopy type of the cofiber \( V/V^\times \).

  This construction promotes to a monoidal functor defined on the groupoid of finite-rank vector bundles
  \[
    \Th \colon (\Vect(S)^\simeq, \oplus) \to (\Sh(S), \otimes)
  \]
  by employing the functoriality of the six-functor formalism (\cref{thm:top6FF})
  sending a correspondence \( S \leftarrow V \to S \) to a linear functor \( q_!q^\ast \) which is left adjoint to \( q_\ast q^! \).
  Since it also takes values in \( \otimes \)-invertible objects, the functor \( \Th \) extends to a monoidal functor defined on virtual vector bundles.

  Next, consider a linear Fredholm map \( l\colon \mathcal{H}' \to \mathcal{H} \) between Banach vector bundles over \( S \) and let \( r \colon \mathcal{H}' \to S \) denote the projection.
  We let \( r_\ast l^!(\unit) \) be suggestively denoted by \( \Th(\ind(l)) \), which we refer to as the \emph{Thom spectrum sheaf of families index}.
\end{construction}
\begin{remark}\label{rem:familiesindex}
  We here remark and justify our terminology.
  If \( S \) is at least compact (Hausdorff), then the functor \[
  \Th(\ind({-}))\colon \pi_0(\Fred(S)^\simeq) \to \Sh(S)
  \]
  does factor through the Atiyah--Singer families index, which takes values in \( \pi_0 \mathit{KO}(S) \).
  To see this, we observe that if \(S\) is compact, then there exists a finite-rank subbundle \( W \) of \( \mathcal{H} \) over \( S \) that is transverse to \( l \) in the sense that the image of \( l \) and \( W \) span \( \mathcal{H} \). 
  Thus, \( V\coloneqq W \times_{\mathcal{H}}\mathcal{H}' \) is a finite-rank subbundle of \( \mathcal{H}' \) due to transversality and being Fredholm.
  Then \( l \) factors as \( l=pe \) as in \cref{lem:factorizaton} where \( e\colon \mathcal{H}' \to W\times_S \mathcal{H}' \) is the zero-section and \( p \) has kernel \( V \).
  By \cref{lem:PicardUpperShriek}, \( l^!(\unit)=e^!p^!(\unit) \) is isomorphic to \( e^!(\unit)\otimes e^\ast p^!(\unit) \), which in turn is identified with \( r^\ast\Th(-W) \otimes \Th(e^\ast T_p) \simeq r^\ast\Th(V-W) \).
  Therefore, \( r_\ast l^!(\unit) \) is isomorphic to the Thom spectrum sheaf of the virtual vector bundle \( V-W \), which is by definition the Atiyah--Singer index \( \ind(l)\in \pi_0 \mathit{KO}(S) \).
\end{remark}

\begin{corollary}
  [Descent to local systems]\label{cor:BFlcnst}
  Let \(f\) be as in \cref{lem:factorizaton}.
  Then
  \(f^!(\unit)\) is locally constant and \(\otimes\)-invertible.
  Therefore, the map \(\BF_f\) lies over the full subcategory of locally constant sheaves on \(S\).
\end{corollary}
\begin{proof}
  Taking a local decomposition \( fj_U=pe \) provided by \cref{lem:factorizaton} and argue as in \cref{rem:familiesindex}, the sheaf \(j_U^\ast f^!(\unit)\) is locally of the form \(
  \Th(-N_e) \otimes \Th(e^\ast T_p),
  \)
  which is the Thom spectrum sheaf of a virtual vector bundle and thus a locally constant and invertible sheaf.

  Then by \cref{lem:VBinvariance}, \(r_\ast f^!(\unit)\) is equivalent to \(x^\ast f^!(\unit)\) for, say, the zero section \(x\colon S \to L\). In particular, it is also locally constant.
\end{proof}

The preceding corollary may be pleasant for the reason that over a homotopically well-behaved topological space \(S\), locally constant sheaves correspond to local systems, also known as \emph{parameterized spectra} due to \cref{thm:monodromy}.

\begin{construction}\label{rem:ExplicitLinearization}
  Let \( f \) be as in \cref{lem:factorizaton} and
  assume further for simplicity that any differentials at two distinct points differ only by a compact linear operator.
  Let \( df \) denote the differential \( L\times_S L \to L\times_S Y \),
  which is a linear map between Banach vector bundles
  \( p\colon L\times_S L \to L \) and \( L\times_S Y \to L \).
  We construct an isomorphism \( f^!(\unit) \simeq p_\ast (df)^!(\unit) \) as follows.
  Define a map \[
    \phi \colon [0,1] \times L \times_S L \to [0,1] \times L \times_S Y
  \]
  by the formula \( \phi(t,x,v) = (t, (1-t)f(v) + t (d_x f)v) \).
  By our assumption on \( f \), \( \phi \) has Fredholm differentials at each point.
  We repeat the above arguments for \(\phi\) to conclude that \(\phi^!(\unit)\) is locally constant
  (and hence constant along the \( [0,1] \)-direction by \cref{lem:globallyconstant}), so that
  \[
    (df)^!(\unit) \simeq i_1^\ast \phi^!(\unit) \simeq i_0^\ast \phi^!(\unit) \simeq p^\ast f^!(\unit)
  \]
  Here we used \cref{lem:basecahgeofdualizingsheaves} in the first and the third identifications.
\end{construction}

\begin{corollary}
    [Linearization hypothesis]
    Let \(f\) be as in \cref{rem:ExplicitLinearization}.
    Then
    \(f^!(\unit)\) can be identified with
    \(\Th(\ind(df))\) the Thom spectrum sheaf of the index of the Fredholm operator \(df\).
    Taking a section \(x\colon S \to L\) of \(r\), the \(r_\ast f^!(\unit)\) can be identified with \(\Th(\ind(d_xf))\).
    The latter identification works as well if only \(f\) is locally proper and 
    \(C^1\)-differentiable near the section \(x\colon S \to L\) with Fredholm differentials.
\end{corollary}
\begin{proof}
    The fist claim follows from the deformation argument given in \cref{rem:ExplicitLinearization}.
    The second claim follows since \( r_\ast f^!(\unit) \simeq x^\ast f^!(\unit) \) again by \cref{lem:VBinvariance}.
\end{proof}

Thus far, by choosing a section \(x\colon S \to L\), we have observed that our Bauer--Furuta map can be seen as a map of the form
\[\BF_f\colon \S^{\ind(d_xf)} \to \S\]
in \(\LocSys(S;\Sp)\) (for a nice base \(S\)).
Later, we will provide a way to compare our construction to the classical Bauer--Furuta construction given by the finite-dimensional approximations as in \eqref{eq:BFintro}. This will be discussed in \cref{thm:approx}.
It is important to note that the finite-dimensionally approximated map \eqref{eq:fin-dimApprox} itself may not be a proper map. So we would like to provide a variant of our Bauer--Furuta map which does not assume the map to be proper.

\subsection{Some purity results for BF}
Here we are going to explore how a smaller subspace of \(L\) will suffice to construct the Bauer--Furuta map of \cref{const:BF}.

\begin{notation}
Let a map \(f\) in the following be locally proper 
and fix a section \(s\colon S \to Y\) of \(q\). Consider the following pullback of topological spaces over \(S\).
\begin{equation}\label{eq:moduli}
    \begin{tikzcd}[column sep=large]
        M \ar[rd, phantom, "\lrcorner", very near start] \ar[r, "\pi"] \ar[d, "i"'] & S \ar[d, "s"'] \ar[dd, bend left, equal] \\
        L \ar[r, "f"] \ar[dr, "r"'] & Y \ar[d, "q"'] \\
        & S
    \end{tikzcd}
\end{equation}
\end{notation}

The topological space \(M\) can be thought of as the moduli space, i.e., an affine cover of the moduli stack%
\footnote{The term \emph{moduli stack} refers to some stacky quotient by the gauge group action.},
of solutions to the equation \(f=s\).

\begin{remark}\label{rem:sectionclosed}
    Since all spaces are assumed to be Hausdorff (\cref{conv:allHaus}), the sections \(s\) and \(i\) are closed immersions.
\end{remark}

In the case when the (affine cover of the) moduli (stack) is proper, we have a variant of \cref{const:BF}.

\begin{construction}
    Assume that the map \(\pi\) in \eqref{eq:moduli} is proper.
    Construct another map as follows.
    \begin{equation}\label{eq:BFmoduliInt}
    \begin{tikzcd}[column sep=large]
        r_\ast f^!(\unit) \ar[r, "r_\ast f^!(\mathrm{unit})"]
        & r_\ast f^! s_\ast (\unit) \ar[r, dash, "\sim", "\mathrm{}"']
        & r_\ast i_\ast \pi^! (\unit) = \pi_\ast \pi^! (\unit) \ar[r, "\mathrm{counit}"] 
        & \unit \ar[r, "\mathrm{unit}"] & q_\ast(\unit)
    \end{tikzcd}
    \end{equation}
    Here the second map is the isomorphism \( f^!s_\ast \simeq i_\ast \pi^! \) obtained by the basechange isomorphism \( s^\ast f_!\simeq \pi_! i^\ast \).
\end{construction}

If \( q \) is contractible, then the unit map \( \unit \to q_\ast \unit \) is inverse to the map \( q_\ast(\mathrm{unit}) \colon \unit \to q_\ast \unit \) (by \cref{cor:contractible}).
Thus, if moreover \( f \) is proper, then the map \( \BF_f \) and the map \eqref{eq:BFmoduliInt} coincide by \cref{rem:BFisPBstable}.

\begin{corollary}\label{prop:modInt}
    If in the diagram \eqref{eq:moduli} \( f \) is proper and \( q \) is contractible, then the map \( \BF_f \) in \cref{const:BF} and the map \eqref{eq:BFmoduliInt} coincide.
\end{corollary}

\begin{remark}
    Assuming that \( \pi \) is proper, we have yet another variant 
    \[
        \begin{tikzcd}[column sep=large]
            r_\ast f^!(\unit) \ar[r, "r_\ast(\mathrm{unit})"] & r_\ast i_\ast i^\ast f^!(\unit) = \pi_\ast i^\ast f^!(\unit) \ar[r, "\sim", "\mathrm{}"', leftarrow] & s^\ast f_! f^! (\unit) \ar[r, "s^\ast(\mathrm{counit})"] & s^\ast(\unit) = \unit \ar[r, "\mathrm{unit}"] & q_\ast(\unit)
        \end{tikzcd}
    \]
    where the second map is the basechange isomorphism \( \pi_\ast i^\ast \simeq s^\ast f_! \).
    By some simple diagram chasing, one can show that the map coincides with \eqref{eq:BFmoduliInt}, also with \cref{const:BF} under the same assumption as above.
\end{remark}

The map \eqref{eq:BFmoduliInt} 
can be interpreted as the restriction map onto the Borel--Moore \(\S\)-homology of the moduli \(M\), followed by the proper pushforward, i.e., the integration map along \(\pi\).
This is a geometric interpretation of our Bauer--Furuta map.

In the Seiberg--Witten map, the choice of a section \(s\) can be thought of a perturbation of the moduli space \(M\); generically \(s\) can be transverse to the map \(f\).
The propositions above thus imply the following.

\begin{corollary}
    [Independence of the perturbation]
    Assume that \(q\) is contractible.
    If \(f\) is proper, then the map \eqref{eq:BFmoduliInt} is independent of the choice of a section \(s\).
\end{corollary}

We may, therefore, abuse notation and write \(\BF_{f}\) for the map \(r_\ast f^!(\unit) \to \unit\) as in \eqref{eq:BFmoduliInt} whenever we are given a section \(s\) with \(f\) being only locally proper and \(\pi\) being proper. As we have just observed, it does not depend on the choice of a section when \(f\) is proper.

Another obvious corollary is that in the construction \eqref{eq:BFmoduliInt}, only the values of \(f\) near the moduli \(M\) is crucial.
This is why we consider these comparison results as purity phenomena.
We record this observation as follows.

\begin{corollary}
    [A purity for the Bauer--Furuta map] \label{cor:BFpurity}
    Assume, in the diagram \eqref{eq:moduli}, that \(\pi\) is proper.
    Take an open subspace \(j \colon N \subset L\) and let \(f_N\) denote the restriction \(N \subset L \xrightarrow{f} Y\).
    Suppose that the map \(N\cap M \hookrightarrow M \xrightarrow{\pi} S\) is proper. Then we have a (commutative) diagram of the following form.
    \begin{equation*}
        \begin{tikzcd}[column sep=small, row sep=small]
            r_\ast f^!(\unit) \ar[rrr, "\BF_f"] \ar[d, end anchor=north] &&& \unit \\
            r_\ast j_\ast j^\ast f^!(\unit) \ar[r, phantom, "="] & r_\ast j_\ast f_N^!(\unit) \ar[rru, "\BF_{f_N}"', start anchor=east] 
        \end{tikzcd}
    \end{equation*}
    If, moreover, \(f^!(\unit)\) is constant along \(r\) and \(j\) is a homotopy equivalence over the base, then the left vertical map is an isomorphism and thus we conclude that \(\BF_f \simeq \BF_{f_N}\).
\end{corollary}

\subsection{Comparison to finite-dimensional approximations}

We end this section by proving a promised comparison result to the classical Bauer--Furuta construction given by the finite-dimensional approximations.

Let us focus on the Seiberg--Witten map, so that we assume that \(f\) is as in \cref{eg:familySWmap}. We further assume for simplicity that our base \(S\) is a point. So \(L\) and \(Y\) are simply Banach vector spaces.

For any topological space \(B\), we let \(p_B\) denote the unique map \(B \to \ast\) onto a point.
For any map \(h\colon B \to S\) and any subspace \(C\subset B\), we write \(h_C\) for the restriction of \(h\) to \(C\).

The proof given below is very similar to that of~\cite[Lemma 2.3 (3)]{BF1}.

\begin{theorem}
    [Finite-dimensional approximation] \label{thm:approx}
    Let \(f\colon L \to Y\) be a proper and \(C^1\)-differentiable map between Banach vector spaces, which we assume to be of the form \(f=l+c\) where \(l\) is a linear Fredholm map and the differentials of \(c\) are compact operators.
        
    Assume that there exists a finite-dimensional linear subspace \(\iota \colon W\subset Y\) that spans \(Y\) together with \(\im(l)\) and admits a splitting \(\pr_W \colon Y \to W\). Let \(W'\) denote \(l^{-1}(W)\), which is also finite-dimensional. \begin{equation*}
        \begin{tikzcd}
            W' \ar[r, "l'"] \ar[d, "\iota'"'] \ar[rd, phantom, very near start, "\lrcorner"] & W \ar[d, "\iota"'] \\
            L \ar[r, "l"] & Y
        \end{tikzcd}
    \end{equation*}
    Consider the composite \(%\)
    g\colon W' \hookrightarrow L \xrightarrow{f} Y \xrightarrow{\pr_W} W %
    \). 
    Since \(M\coloneqq f^{-1}(0)\) is compact, it is bounded in \(L\). Let us assume further that there exists a bounded open subspace \(j\colon N \hookrightarrow L\) containing \(f^{-1}(0)\) such that \(N\cap f^{-1} \pr_W^{-1}(0)\) is compact.\footnote{Note that all the assumptions are satisfied by the Seiberg--Witten map by~\cite[2.2, 2.3, 3.1]{BF1}.}\begin{equation*}
        \begin{tikzcd}
            \mathllap{g\colon} W' \ar[r, "\iota'"] &L \ar[r, "f"] & Y \ar[r, "\pr_W"] &W \\
            N' \ar[u, "j'"] \ar[r, "\iota''"'] \ar[ru, phantom, very near start, "\urcorner"] & N \ar[u, "j", hookrightarrow]
        \end{tikzcd}
    \end{equation*}
    In particular, the map \(g_{N'}\) has compact zero-set and the construction \eqref{eq:BFmoduliInt} makes sense.

    Then the map \(\BF_f \colon \S^{\ind(l)} \simeq p_L{}_\ast f^!(\unit) \to \S\) can be identified with the following composite map, which can be thought of as the classical Bauer--Furuta map obtained by the finite-dimensional approximation argument:
    \[
    \BF_f^{\mathrm{apprx}} \colon
    p_{W'\ast} g^!(\unit) \to p_{N'\ast}{} j'^\ast g^!(\unit) = p_{N'\ast}{} g^!_{N'}(\unit) \xrightarrow{\BF_{g_{N'}}} \S
    \]
\end{theorem}

\begin{proof}
    Let \(\mathcal{H} = \ker(\pr_W)\). We first choose an isomorphism \(L \cong \mathcal{H} \times W'\) such that the map \(\mathcal{H} \times W' \cong L \xrightarrow{l} Y = \mathcal{H} \times W\) is of the form \(\begin{pmatrix}
        1 & 0 \\ 0 & \ast
    \end{pmatrix}\), which exists since \(l\) is Fredholm. Define a map \(g'\) as follows.
    \[
    \begin{tikzcd}[row sep=small]
        \mathllap{\pr_{W'}^{-1}(N')=:\,}N'' \ar[d, dash, "\cong"] &\mathllap{g'\colon} L \ar[r] \ar[d, dash, "\cong"] & Y \ar[d, equal] \\
        \mathcal{H} \times N' \ar[r, hookrightarrow] \ar[d] \ar[rd, phantom, very near start, "\lrcorner"] &\mathcal{H}\times W' \ar[r] \ar[d, "\pr"'] \ar[rd, phantom, very near start, "\lrcorner"] & \mathcal{H} \times W \ar[d, "\pr_W"] \\
        N' \ar[r, phantom, "\subset"] &W' \ar[r, "g"] & W
    \end{tikzcd}
    \] In other words, \(g'=l+\pr_W\circ c \circ \pr_{W'}\). The map \(g'\) also has differentials of the form \(l+\textit{compact}\) as well as \(f\), and \(g'_{N''}\) has compact zero-set as well as \(g_{N'}\). By the pullback-stability (\cref{rem:BFisPBstable}) we conclude that \(\BF_{g'_{N''}} \simeq \BF_{g_{N'}}\).

    Construct the map over the base space \([0,1]\) \[
    \begin{tikzcd}[column sep=small, row sep=tiny] \psi \colon L \times [0,1] \ar[rr] \ar[rd] && Y \times [0,1] \ar[ld] \\
    & {[0,1]} &
    \end{tikzcd}
    \] by the formula \(\psi(-,t) = \left(\left(1-t\right)f + t(l+\pr_W\circ c), t\right)\). The vertical differentials of \(\psi\) are of the form \(l+\textit{compact}\) hence Fredholm, and the zeros of \(\psi_{N\times[0,1]}\) are contained in \(N\cap f^{-1}\pr_W^{-1}(0)\) hence compact.
    Similarly, we define another \(1\)-parameter family \(\phi\) by the formula \(\phi(-,t) = (l + \pr_W\circ c \circ((1-t)\id + t \pr_{W'}), t)\). It has Fredholm differentials, and the zeros of \(\phi_{N\times [0,1]}\) are contained in the compact \(N\cap f^{-1}\pr_W^{-1}(0)\).
    We now execute the construction \eqref{eq:BFmoduliInt} to the maps \(\psi\) and \(\phi\) and then invoke locally constancy (\cref{cor:BFlcnst}) and homotopy invariance (\cref{lem:globallyconstant}) to conclude that \(\BF_{f_N} \simeq \BF_{g'_N}\).

    Finally, we have the purity triangles (\cref{cor:BFpurity}) of the form \(\BF_{f} \to \BF_{f_N}\) and \(\BF_{g'_N} \to \BF_{g'_{N\cap N''}} \leftarrow \BF_{g'_{N''}}\), which we consider as the maps in the slice category \(\Sp_{/\S}\).

    Combining all together, we get a zigzag connecting the desired maps. \[
    \begin{tikzcd}[row sep=small, column sep=large]
        \mathllap{\S^{\ind(l)}\simeq\;} \bullet \ar[d] \ar[r, "\BF_f"] & \S \\
        \mbox{} & \\
        \mathllap{\S^{\ind(l)}\simeq\;} \bullet \ar[u] \ar[ruu, "\BF^{\mathrm{apprx}}"']
    \end{tikzcd}
    \] We also observe that identifications of the left hand sides with \(\S^{\ind(l)}\) can be given as in \cref{rem:ExplicitLinearization}, which commute with the vertical zigzag that we have constructed so far. Therefore, we have completed constructing a homotopy \(\BF_f \simeq \BF_f^{\mathrm{apprx}}\).
\end{proof}

We do not think that this theorem will serve a substantial role in our formalism, as we believe that all the basic behaviors satisfied by the Bauer--Furuta invariants can be proved in a simpler way within our definition. Nonetheless, the theorem serves at least as a justification that our construction subsumes the constructions of~\cite{Fur01} and~\cite{BF1}.

\section{Genuine equivariance} \label{section:SHG}

In this section, we are going to provide a portion (\cref{dfn:SH}) of genuine equivariant six-functor formalism on topological spaces with Lie group actions, which is sufficient to define the genuine equivariant Bauer--Furuta map (\cref{const:genuineBF}).

\subsection{A quick review of genuine equivariant homotopy theory}
Let \(G\) be a Lie group.
The category of \emph{genuine (proper equivariant) \(G\)-spectra} \(\Sp^G\) is defined to be the presheaves on the (homotopy-coherent nerve of the) topological category of \(G\)-orbits \(G/K\) with compact stabilizer subgroups, while the representation spheres \(S^V\) are formally inverted.
It fits into a larger category of \emph{global spectra} \(\Glo\Sp\), where an orbit is replaced by a topological stack \(BK\) for all compact Lie groups \(K\).

By the work of~\cite{globalspaces}, these constructions are significantly simplified and it is shown that \[
\Glo\an \simeq \Sh_{\AA^1}(\mathsf{SepSt}).
\]
Here, \(\mathsf{SepSt}\) is the full subcategory of \(\Sh(\Mfd;\an_{\le1})\) spanned by those stacks satisfying appropriate representability and separated conditions and the subscript \({\AA^1}\) indicates the homotopy invariant sheaves.

For a \(G\)-manifold \(M\) with proper action, let \(M{\sslash} G \in \mathsf{SepSt}\) denote the quotient stack. Motivated by the above reformulation of global homotopy theory, and also inspired by the definition of equivariant motivic stable homotopy theory~\cite{Hoyois}, it is reasonable to consider the following category as a genuine equivariant spectral sheaf theory. \[
    \SH(M{\sslash}G) \coloneqq \Sh_{\AA^1}(\Sm^{\mathrm{sep}}_{M{\sslash}G})[\S^{-V}]
\]
Here, the category \(\Sm^{\mathrm{sep}}_{M{\sslash}G}\) consists of separated stacks over \(M{\sslash}G\) whose structure morphisms onto \(M{\sslash}G\) are \emph{representable submersion{s}}.
These categories are studied extensively in~\cite[Part II]{BastiaanPhD}.

It is possible to unwind the definition and redefine this category, and adopt it even to non-differentiable base spaces.
In the following, we take an abstract approach to these categories. Specifically, we will provide a simple axiom that such a topological variant of the site \(\Sm^{\mathrm{sep}}_{X{\sslash}G}\) should satisfy, and show that these axioms are sufficient to prove the proper basechange theorem.
A concrete (six-functor) formalism that satisfies our requirement will be studied in a future work, as well as the detailed behavior of the genuine equivariant Bauer--Furuta map.

\subsection{Genuine equivariant homotopy theory}

Fix a Lie group \( G \). Recall that a \( G \)-action on a topological space \( X \) is proper if and only if the stacky quotient \( X{\sslash}G \) is separated.

\begin{definition}
We fix a collection of \(G\)-equivariant continuous maps between topological spaces, which is referred to as the collection of \emph{\(G\)-smooth maps}, that satisfies the following axiom.
\begin{enumerate}
    \item Every \(G\)-smooth map is an open map.
    \item If a \(G\)-map \(p\colon Y \to X\) admits an open cover \(\left\{U_i\to Y\right\}\) consisting of \(G\)-invariant open subspaces such that each map \(U_i \to p(U_i)\) is \(G\)-smooth, then \(p\) is \(G\)-smooth.
    \item The collection of \(G\)-smooth maps is closed under basechange and composition, and it contains all \(G\)-equivariant homeomorphisms.
    \item Every (finite-rank) \(G\)-equivariant vector bundle is \(G\)-smooth.
    \item If \( K \) is a compact subgroup of \( G \), \(H\) is a closed subgroup of \( G \) with \( K\subset H \), and \(S\) is a topological space with \(H\)-action, then the map \(\ind_K^G(S) \to \ind_H^G(S)\) is \(G\)-smooth.
    \item For each \(X\), the collection of \(G\)-smooth maps with target \(X\) forms a small category.
\end{enumerate}

Define the site \( \Sm_X^G \) as follows.
The objects are the \( G \)-smooth maps with target \( X \), and the morphisms are \( G \)-maps over \( X \). The category \( \Sm_X^G \) is equipped with the Grothendieck topology generated by open coverings consisting of \( G \)-invariant open subspaces.
\end{definition}

For example, the category \( \Sm^G \coloneqq \Sm^G_{\pt} \), where \( X \) is taken to be a point, contains every differentiable \( G \)-manifold with proper action, such as an orbit \( G/K \) with compact stabilizer  subgroup \( K \). This coincides with the convention from \emph{proper equivariant} homotopy theory.

An easy example of a class of \( G \)-smooth maps is the class of \( G \)-maps whose underlying continuous maps are topological submersions,
which could be one of the worst choice due to the absence of local models of \( G \)-smooth maps.
Thus, a preferred choice for the definition of \(G\)-smooth maps would be the smallest collection satisfying the axiom given above. However, the precise definition will not be needed in the following arguments.

\begin{remark}[Smooth/open sharp] \label{rem:Gsharp}
    Let \( p\colon Y \to X \) be \( G \)-smooth and let \(p_\circ\) denote the functor \(\Sm_Y^G \to \Sm_X^G\) obtained by post-composing with \(p\). Then we have an adjunction
    \begin{math}
        \begin{tikzcd}
            \Sm_Y^G \ar[r, "p_\circ", shift left] & \ar[l, "p^\ast", shift left] \Sm_X^G,
        \end{tikzcd}
    \end{math} each of which is a morphism of sites, meaning that the precomposition functors preserve sheaves. Thus, we obtain a triple as follows
    \[
        \begin{tikzcd}[column sep = large]
            \Sh(\Sm_Y^G) \ar[r, shift left = 5, "p_\sharp", ""'{name = LL}] \ar[r, shift right = 5, "p_\ast"', ""{name = R}] & \Sh(\Sm_X^G) \ar[l, "p^\ast"{description, name = L}]
            \ar[from = LL, to = L, phantom, "\dashv"{rotate = -90}]
            \ar[from = L, to = R, phantom, "\dashv"{rotate = -90}]
        \end{tikzcd}
    \] where \(p_\sharp\) is the sheafified left Kan extension functor along \(p_\circ\).
    Next, consider the site \( \Op_Y^G \) of \( G \)-invariant open subsets
    and the functor \(\jmath \colon \Op_Y^G \hookrightarrow \Sm_Y^G\), which is a morphism of sites having the covering lifting property\footnote{See~\cite[Appendix A.1]{Syn_E} for definition and consequence.}.
    By this, we obtain the following triple
    \[
        \begin{tikzcd}[column sep = large]
            \Sh(\Op_Y^G) \ar[r, shift left = 5, "\jmath_\sharp", ""'{name = LL}] \ar[r, shift right = 5, "\jmath_\ast"', ""{name = R}] & \Sh(\Sm_Y^G) \ar[l, "\jmath^\ast"{description, name = L}]
            \ar[from = LL, to = L, phantom, "\dashv"{rotate = -90}]
            \ar[from = L, to = R, phantom, "\dashv"{rotate = -90}]
        \end{tikzcd}
    \] where \(\jmath_\sharp\) is the sheafified left Kan extension along \(\jmath\) and \(\jmath^\ast\) is the restriction along \(\jmath\).

    We would like to abuse notation so that \(p_\sharp \colon \Sh(\Op_Y^G) \to \Sh(\Sm_X^G)\) also denotes the composite \(p_\sharp \jmath_\sharp\).
\end{remark}
\begin{lemma}
    [Smooth basechange] \label{lem:GSBC}
    Let \( f\colon X' \to X \) be a \( G \)-map and \( p\colon Y \to X \) be \( G \)-smooth.
    Let \(Y'\) denote the pullback \(Y \times_{X} X'\) and let \(f'\) and \(p'\) be the resulting basechanges.
    Then the following diagram commutes.
    \[
        \begin{tikzcd}
            \Sh(\Op_{Y'}^G) \ar[r, "p'_\sharp"] \ar[dr, phantom, "\Rightarrow"{rotate = -35}] & \Sh(\Sm_{X'}^G) \\
            \Sh(\Op_Y^G) \ar[r, "p_\sharp"'] \ar[u, "f'^\ast"]	& \Sh(\Sm_X^G) \ar[u, "f^\ast"']
        \end{tikzcd}
    \]
\end{lemma}
\begin{proof}
    Since every functor in the diagram commutes with colimits, it suffices to show that the map \[ p'_\sharp f'^\ast j_!(\unit) \to f^\ast p_\sharp j_!(\unit) \] is an isomorphism for any \( G \)-invariant open embedding \( j\colon U \to Y \).
    This follows from the following easy commutative diagram.
    \[
        \begin{tikzcd}
            \Op_{Y'}^G \ar[r, "p'(-)"] & \Sm_{X'}^G \\
            \Op_Y^G \ar[r, "p(-)"'] \ar[u, "f'^{-1}"]	& \Sm_X^G \ar[u, "f^{-1}"']
        \end{tikzcd}
    \]
\end{proof}

\begin{proposition}[Proper pushforward] \label{prop:Gproperpushforward}
    Let \( f\colon X' \to X \) be a proper \( G \)-map. We assume that \( G \)-action on \( X' \) is a proper action.
    Then the functor \( f_\ast \colon \Sh(\Sm_{X'}^G) \to \Sh(\Sm_X^G) \) preserves colimits.
    
    Moreover, the proper projection formula holds, i.e., the map \( f_\ast(-) \otimes (-) \to f_\ast(-\otimes f^\ast(-)) \) is an isomorphism.
\end{proposition}
\begin{proof}
    Since the functors \( \left\{p^\ast \colon \Sh(\Sm_X^G) \to \Sh(\Op_Y^G)\right\}_{p\in\Sm_X^G} \) are jointly conservative and each \( p^\ast \) preserves colimits, it suffices to show that each \( p^\ast f_\ast \) is colimit-preserving. Fix such a \(G\)-smooth map \(p\) and the following cartesian square.
    \begin{equation*}
    \begin{tikzcd}[row sep=small, column sep=small]
    Y' \ar[r, "p'"] \ar[d, "f'"'] \ar[rd, phantom, very near start, "\lrcorner"]	&	X' \ar[d, "f"]
    \\
    Y \ar[r, "p"']	& X
    \end{tikzcd}
    \end{equation*}

    By \cref{lem:GSBC}, \( p^\ast f_\ast \) is equivalent to the following composite functor. \[
        \begin{tikzcd}
            \Sh(\Sm_{X'}^G) \ar[r, "p'^\ast"] & \Sh(\Op_{Y'}^G) \ar[r, "f'_\ast"] & \Sh(\Op_Y^G)
        \end{tikzcd}
    \]
    Since \( p'^\ast \) preserves colimits, it suffices to show that \( f'_\ast \) preserves colimits.
    The functor \( f'_\ast \) is the same as the functor \( (f'{/}G)_\ast \colon \Sh(Y'/G) \to \Sh(Y/G) \) on the quotient topological spaces, and the map \( f'{/}G \) is separated universally closed by assumption. Therefore, it preserves colimits.

    The projection formula can be proved in the same way, by observing that the functors \( p^\ast \) and \( p'^\ast \) preserve tensor products.
\end{proof}

\begin{definition}
    Let \( X \) be a topological space with \( G \)-action.
    Define the full subcategory \( \Sh_{\AA^1}(\Sm_X^G) \) of \( \Sh(\Sm_X^G) \) to be the localization of \( \Sh(\Sm_X^G) \in \CAlg(\Pr^\L_\st) \) with respect to the \( \otimes \)-ideal generated by the map \( X\times\RR^1 \to X \) with trivial \( G \)-action on \( \RR^1 \).
\end{definition}

For each (finite-dimensional) \( G \)-representation \( V \), consider the map \( p\colon V \to \ast \) and the zero-section \( e\colon \ast \to V \).
Then we have \( p_\sharp s_\ast(\unit) \in \Sh(\Sm^G) \), which is thought of as the Thom spectrum of the form \( \S^V \).

\begin{definition}\label{dfn:SH}
    Define the \emph{genuine equivariant homotopy category} \( \SH^G_\topl \) to be the formal inversion
    \[ \Sh_{\AA^1}(\Sm^G)\left[p_\sharp s_\ast(\unit)^{-1}\mid V\in \Rep(G)\right] \; \in \CAlg(\Pr^\L). \]

    For a topological \( G \)-space \( X \), define its genuine equivariant stable homotopy category \( \SH^G_\topl(X) \) to be \( \Sh_{\AA^1}(\Sm_X^G) \otimes_{\Sh_{\AA^1}(\Sm^G)} \SH^G_\topl \).

    For a \( G \)-map \( f \colon X' \to X \), let \( f^\ast \colon \SH^G_\topl(X) \to \SH^G_\topl(X') \) denote the functor obtained by tensoring the functor \( f^\ast  \colon \Sh_{\AA^1}(\Sm_{X'}^G) \to \Sh_{\AA^1}(\Sm_X^G) \) with \( \SH^G_\topl \).

    For a proper \( G \)-map \( f \colon X' \to X \), by \cref{prop:Gproperpushforward}, the functor \( f^\ast \) is an internal left adjoint functor in \( \Mod_{\Sh_{\AA^1}(\Sm^G)}(\Pr^\L) \). Therefore, for such a proper \( G \)-map, \( f^\ast \) on \( \SH^G_\topl(-) \) is an internal left adjoint functor, whose right adjoint \( f_\ast \) has a further right adjoint which is denoted by \( f^! \).
\end{definition}

The formal inversion is essential to ensure that the upper shriek functors compute correct Thom spectra. 
We finally have the genuine equivariant analogue of \cref{const:BF}.

\begin{construction}\label{const:genuineBF}
    Let \( L \), \( X \) and \( S \) be topological spaces acted on by a Lie group \( G \) where we assume \(L\) to have proper \(G\)-action.
    Let \(f \colon L \to X\) be a proper \(G\)-map over \(S\).
  \begin{equation*}
      \begin{tikzcd}[row sep = small]
          L \ar[rr, "f"] \ar[rd, "r"'] & & X \ar[dl, "q"] \\
          & S &
      \end{tikzcd}
  \end{equation*}
  Then we have a map in \( \SH^G_\topl(S) \) of the following form.
  \[\BF_f \colon r_\ast f^!(\unit) = q_\ast f_\ast f^!(\unit) \xrightarrow{q_\ast(\mathrm{counit})} q_\ast(\unit)\]
\end{construction}

This is our definition of the \emph{genuine equivariant Bauer--Furuta map}.

\appendix
\section{The Seiberg--Witten map} \label{section:SW}

\subsection{Recollection of the Seiberg--Witten map}
Let us recall the construction of the Seiberg--Witten map \eqref{eq:SWmap}, following~\cite{Fur01} and~\cite{Fur97}.

In this section, \(\Mfd\) denotes the category of \(C^\oo\)-manifolds and \(C^\oo\)-maps between them. We equip \(\Mfd\) with the Grothendieck topology generated by open covers. The category \(\St\coloneqq \Sh(\Mfd;\an_{\le1})\) of (\(1\)-truncated) stacks will be useful. For example, we have a fully faithful embedding \(B\colon \Lie\Grp \subset \Grp(\St) \hookrightarrow \St_\ast\) given by the delooping functor.
All Lie groups, vector spaces and manifolds will be viewed as objects in \(\St\).
In the slice category \(\St_{/\mathcal{Y}}\) over a stack \(\mathcal{Y}\), we have \(\hom_\mathcal{Y}(-,-)\in\St_{/\mathcal{Y}}\) the internal hom and \(\hom_{/\mathcal{Y}}(-,-)\in\St\) the \(\St\)-enrichment.

Consider the Lie group \(\Spin^c(n)\), which can be defined as \(\left(\Spin(n){\times}\U(1)\right)/\diag\{\pm1\}\).
In dimension \(4\), we have an isomorphism \(\Spin(4) \cong \USp(1)\times\USp(1)\) which is exhibited by the representation \(\USp(1)\times\USp(1) \to \SO(4)\) given as follows.
\[
\USp(1)\times\USp(1)\ni \; (q_-, q_+) \quad \mapsto \quad q_-({-})q_+^{-1} \; \curvearrowright \HH \; (\cong \RR^4)
\]
The group \(\Spin^c(4)\) acts in the same way (where the \(\U(1)\)-part acts trivially) and the resulting representation is denoted by \(_-\HH_+\). We further need the following three representations of \(\Spin^c(4)\).
\begin{equation*}\begin{aligned}
    &_+ \HH \phantom{_\CC}\colon  \; \USp(1) {\times} \USp(1) {\times} \U(1) \ni & (q_-, q_+ , u) &\;\mapsto q_+({-})z^{-1} \curvearrowright \HH \\
    &_- \HH \phantom{_\CC}\colon   & (q_-, q_+ , u) &\;\mapsto q_-({-})z^{-1} \curvearrowright \HH \\
    &_+ \HH_+ \colon  & (q_-, q_+ , u) &\;\mapsto q_+({-})q_+^{-1} \curvearrowright \HH
\end{aligned}\end{equation*}
The two \(_{\pm}\HH\) are complex representations, where we let the complex number \(\sqrt{{-}1}\) act on \({}_\pm\HH\) as the right multiplication by \(i\in \HH\).
Using the orthonormal basis \((1,i,j,ji)\) of \(\HH\) as an orientation for \(_-\HH_+\), the imaginary part \(\Im(_+\HH_+)\), which is invariant under \(\Spin^c(4)\), is isomorphic to \(\Lambda^{2,+}(_-\HH_+)\) the \(({+}1)\)-eigenspace of the Hodge star operator acting on \(\Lambda^2\).
The projection map from \(\Lambda^2(_-\HH_+)\) onto \(\Lambda^{2,+}(_-\HH_+)\cong \Im({}_+\HH_+)\) will be denoted by the superscript \(({-})^+\).
We have two (complex and real, respectively) Clifford actions as follows.
\begin{equation}\label{eq:Cl-action}\begin{aligned}
(_-\HH_+ {\otimes_\RR} \CC ) \otimes {}_+\HH &\to {}_-\HH \quad;& (a+b\sqrt{{-}1}, \phi) &\mapsto a\phi + b\phi i  \\
_-\HH_+ \otimes_\RR {}_-\HH_+ &\to {}_+ \HH_+ \;;& (a,b) &\mapsto \overline{a} b 
\end{aligned}\end{equation}
And finally a quadratic form:
\[ \sigma \colon 
_+ \HH \ni \; \phi \quad \mapsto \quad - \phi i \overline{\phi} \otimes \sqrt{-1} \; \in \Lambda^{2,+}(_+\HH_+) {\otimes_\RR} \sqrt{-1} \RR
\]

Let \(X\) be a closed oriented \(4\)-manifold.
Take a riemannian metric \(g\) and a spin\(^c\) structure \(\mathfrak{s}\) on \(X\), i.e., \(\mathfrak{s}\) is a lift as in the following diagram in \(\St\).
\[
\begin{tikzcd}[row sep=small, column sep=large]
    & B{\Spin^c}(4) \ar[d] \\
    & B{\SO}(4) \ar[d] \\
    X \ar[r, "T^\ast_X"'] \ar[ru, "g"{description}, end anchor=west] \ar[ruu, dashed, "\mathfrak{s}"', end anchor=west, bend left] & B{\GL^+}(4)
\end{tikzcd}
\] 
In other words, \(\mathfrak{s}\) is a pair consisting of a \(\Spin^c(4)\)-torsor \(P\) on \(X\) together with an isomorphism of sheaves \(\Omega^1_X \cong \hom^{\Spin^c(4)}(P, {}_-\HH_+)\) where the right hand side is the mapping stack of \(\Spin^c(4)\)-equivariant maps.
Next, we take a spin\(^c\)-connection on the torsor \(P\) such that it lifts the riemannian connection on \(T^\ast_X\) associated with the metric \(g\). Such a connection corresponds exactly to a \(\U(1)\)-connection \(A\) (on the determinant line bundle of \(P\)), since the Lie algebra of \(\Spin^c(4)\) splits as a sum of \(\mathfrak{so}(4)\) and \(\mathfrak{u}(1)=\sqrt{{-}1}\RR\). We abuse notation and let the corresponding \(\Spin^c(4)\)-connection be denoted also by \(A\).

We construct the \emph{Seiberg--Witten map} as follows. The connection \(A+ a\sqrt{{-}1}\) defines the covariant derivative on the spaces such as \(\hom^{\Spin^c(4)}(P, {}_{(\pm)}\HH_{(+)})\) and thus we can post-compose with the Clifford actions~\eqref{eq:Cl-action} to obtain the Dirac-type operator of the following form.
\[
\hom^{\Spin^c(4)}\left(P,\, {}_-\HH_+ \oplus {}_+\HH\right) \to \hom^{\Spin^c(4)}\left(P,\, {}_-\HH \oplus ({}_+\HH_+{\otimes} \sqrt{{-}1})\right)
\]
In fact, it is of the form \((a,\phi) \mapsto \left(D_A\phi + a\phi i,\, d^\ast a \sqrt{-1} + (da)^+ \sqrt{-1}\right)\).
Consider another map on the same spaces given by the formula \[
(a, \phi) \mapsto (0, F_{A}^+ - \sigma(\phi) )
\] where \(F_A\) is the curvature \(2\)-form (valued in \(\sqrt{{-}1}\RR\)).
We also would like to add the locally constant functions on the left hand side, which is mapped to \(\hom^{\Spin^c(4)}(P,{}_+\HH_+)\) via the inclusion \(\RR\cong \Re({}_+\HH_+) \subset {}_+\HH_+\). Combining all together, we obtain the Seiberg--Witten map.
\begin{equation}\begin{aligned}\label{eq:monopole/A}
\mu_{g,\mathfrak{s}, A} \colon \hom^{\Spin^c}(P, {}_-\HH_+ \oplus {}_+\HH) \times H^0_{\dR}(X,\RR) \to \hom^{\Spin^c(4)}(P, {}_-\HH \oplus ({}_+\HH_+{\otimes}\sqrt{{-}1})) \\
(a,\phi,r) \mapsto (D_{A+a\sqrt{{-}1}}\phi,\, r + d^\ast a + F_{A+a\sqrt{{-}1}}^+ + \phi i \overline{\phi} \sqrt{{-}1})
\end{aligned}\end{equation}
For \(k>4\), the map on the Sobolev completions \(L^2_k \to L^2_{k-1}\) induced by the linear part, together with the map on \(L^2_k\)'s induced by the nonlinear part post-composed by the Sobolev embedding \(L^2_k \to L^2_{k{-}1}\), is the map we can apply our Bauer--Furuta construction: See~\cite[Lemma 3.1]{BF1} and \cref{lem:familyProper}.

Note that the Seiberg--Witten map is parameterized over the space of \(\U(1)\)-connections and has the gauge group equivariance, which we summarize in the following language.

\begin{remark}[Gauge group]
    Consider the stack \( B_\nabla {\Spin^c}(n) \) classifying \(\Spin^c(n)\)-torsors and connections on them, which is just the quotient stack of the lie algebra \( \mathfrak{g} \) acted on by \( g\in G \) via the formula \( g^{{-}1}dg + g^{{-}1}({-})g \). Let \( \Spin^c_\nabla(X) \) denote the stack \(\hom_{/B{\GL^+}(4)}(X,B_\nabla{\Spin^c}(4))\),
    which roughly classifies a triple \((g,\mathfrak{s},A)\) consisting of a metric, a spin\( ^c \)-structure and a spin\( ^c \)-connection.
    We similarly consider the stacks \( \mathrm{Met}(X) = \hom_{/B{\GL^+}(4)}(X,B{\SO}(4)) \) and \( \mathrm{Met}_\nabla(X) = \hom_{/B{\GL^+}(4)}(X,B_\nabla{\SO}(4)) \), where the forgetful map \( \mathrm{Met}_\nabla(X) \to \mathrm{Met}(X) \) admits a section \( \mathrm{LC} \colon \mathrm{Met}(X) \to \mathrm{Met}_\nabla(X) \) given by the riemannian connection.
    Form the following pullback and obtain our base stack \( \mathcal{S}_X \) for the Seiberg--Witten map.
    \[
    \begin{tikzcd}
        \mathcal{S}_X  \ar[r, ""] \ar[d, ""'] \ar[rd, phantom, very near start, "\lrcorner"]           	    & \Spin^c_\nabla(X) \ar[d, ""] \\
        \mathrm{Met}(X) \ar[r, "\mathrm{LC}"']    & \mathrm{Met}_\nabla(X)
    \end{tikzcd}
    \]
    Fixing a (metric and a) spin\(^c\)-structure giving a torsor \(P\), the fiber of the forgetful map \(\mathcal{S}_X \to \Spin^c_\nabla(X) \to \Spin^c(X)\) is isomorphic to the moduli stack \(\mathcal{A}(\det P)\) of \(\U(1)\)-connections on the determinant line bundle. The based looping \(\Omega \mathcal{A}(\det P)\) is called the \emph{gauge group} \(\mathcal{G}\), which is isomorphic to the group \(\hom(X,\U(1))\).
    
    Over the product \(X\times\mathcal{S}_X\), we have the \(\Spin^c(4)\)-torsor, say \(\mathcal{P}\), which is pulled-back along the evaluation map \(X{\times}\hom(X,B{\Spin}^c(4))\to B{\Spin^c}(4)\).
    The Seiberg--Witten map \(\mu\) can be considered as the map \begin{equation}\label{eq:monopole/stack}
        \mu \colon \hom^{\Spin^c(4)}_{\mathcal{S}_X}(\mathcal{P},\,{}_-\mathbb{H}_+{\oplus}{}_+\mathbb{H}) \times H^0_\dR(X,\mathbb{R}) \to \hom^{\Spin^c(4)}_{\mathcal{S}_X}(\mathcal{P}, \, {}_-\mathbb{H}\oplus{}_+\mathbb{H}_+{\cdot}\sqrt{{-}1})
    \end{equation} over \(\mathcal{S}_X\), which is defined by the same formula we presented above. Since the stack \(\mathcal{S}_X\) is highly structured, fixing a single spin\(^c\)-structure, the map \(\mu\) restricts to a \(\mathcal{G}\)-equivariant map, for trivial reason.
\end{remark}

Note that we have the Galois action \(\Gal(\mathbb{C}/\mathbb{R}) \to \Aut_{\Grp}(\Spin^c(n))\) given by the complex conjugation on the \(\U(1)\)-part of \(\Spin^c(n)\).
It induces a \(\Gal(\mathbb{C}/\mathbb{R})\)-action on the stack \(\Spin^c_\nabla(X)\). Let \((\mathfrak{s}',A')\) denote the conjugate of \((\mathfrak{s},A)\in\Spin^c_\nabla(X)\) by this \(\Gal(\mathbb{C}/\mathbb{R})\)-action.
This gives us a \(\Pin(2)\)-symmetry of the Seiberg--Witten map \eqref{eq:monopole/stack} as follows.

\begin{remark}[\(\Pin(2)\)-symmetry{~\cite{Fur97}}]
    Define order \(4\) automorphisms \(J\) as follows.
    On the source of \(\mu\), \(J\) acts by the formula \((\mathfrak{s},A;a,\phi,r) \mapsto (\mathfrak{s}',A';-a,\phi j,-r)\), and on the target of \(\mu\) by the formula \((\mathfrak{s},A;\varphi,b) \mapsto (\mathfrak{s}',A';\varphi j,-b)\). The map \(\mu\) commutes with these \(J\)-actions, and \(J\) and \(u\in H^0_\dR(X,U(1))\) satisfy the relations \(J\circ u = \overline{u}\circ J\) and \(J^2=-1\in\U(1)\). Therefore, in this way the Seiberg--Witten map \(\mu\) promotes to a map over the stack \(B{H}^0_\dR(X, \Pin(2))\).

    If we consider a spin\(^c\)-structure \(\mathfrak{s}\) coming from a spin structure, then the determinant line bundle is trivial and admits a trivial connection \(A_0\) which is fixed by the \(\Gal(\mathbb{C}/\mathbb{R})\)-action. Thus, the single Seiberg--Witten map \(\mu_{g,\mathfrak{s},A_0}\) \eqref{eq:monopole/A} promotes to a \(\Pin(2)\)-equivariant map.
\end{remark}

Another way of stating this \( \Pin(2) \)-symmetry is to observe that the representations and maps between them we have used so far are promoted to be equivariant under the larger group \( \Spin^{c-}(4) = \left(\Spin(4) {\times} \Pin(2)\right){/}\diag\{{\pm}1\} \) and to replace the base stack accordingly.
This perspective was introduced by~\cite{Nakamura13}.

The stacky constructions presented above would be readily applicable to our reformulation of the Bauer--Furuta construction \emph{if only} there exists a genuine six-functor formalism, defined over those geometric stacks, which should send \(BG\) of a Lie group to the genuine \(G\)-spectra and subsume the Atiyah duality for \(G\)-manifolds.

\subsection{Some example of properness}
We here deal with an abstract lemma that enables us to conclude the properness of the families Seiberg--Witten map.
Essentially, it is a recollection of the arguments given in~\cite{BF1}.
We continue to assume that all topological spaces are Hausdorff. 
First, observe the following.
\begin{lemma}\label{lem:CGtotalspace}
    Let \(S\) be compactly generated and \(\pi\colon E \to S\) be a Banach vector bundle. Assume that either of the following holds.
    \begin{enumerate}
        \item \(S\) is first countable,
        \item \(S\) is locally compact, or
        \item the local triviality of the fiber bundle \(\pi\) is interpreted in the cartesian closed category of compactly generated topological spaces.
    \end{enumerate}
    Then the total space \(E\) is compactly generated.
\end{lemma}
\begin{proof}
    Since being compactly generated is a local property, 
    we may assume that the total space \(E\) is of the form \(U\times \mathcal{H}\) for some Banach space \(\mathcal{H}\).
    Note that a finite product of first countable topological spaces (such as \(U\) and \(\mathcal{H}\)) is first countable and hence compactly generated.
    Also, a product of a compactly generated space \(\mathcal{H}\) and a locally compact space \(U\) is also compactly generated.
\end{proof}
Of course, in the previous lemma, \(S\) can be a product of a first countable space and a locally compact space. 
We note that the homotopy invariance (\cref{cor:contractible} or \cref{lem:globallyconstant}) for (Banach) vector bundles remains valid even when we replace the product topology by compactly generated ones, though in certain statements such as \cref{lem:locproperpullback}, the choice of product topology seems to be crucial.

We next record a variant of~\cite[Lemma 2.3]{BF1}.
The proof goes basically the same as that of~\cite[Lemma 2.3]{BF1}. This will serve as a source for the families Bauer--Furuta invariant based on our formalism.

\begin{lemma}
    \label{lem:familyProper}
    Let \(f\colon E' \to E\) be a map between Banach vector bundles over a base \(S\) which is of the form \(f=l+c\) for \(l\) a linear Fredholm map and \(c\) a (possibly nonlinear) compact map in the sense that it maps a disk bundle over a compact subset \(K\subset S\) to a relatively compact subset of \(K\times_S E\). Assume that on each fiber the map \(f_x \colon E'_x \to E_x\) has bounded preimages of bounded subsets and that the vector bundle \(q\colon E \to S\) satisfies either of the assumptions in \cref{lem:CGtotalspace}.
    Then \(f\) is a proper map.
\end{lemma}
\begin{proof}
    Since properness is local on the target, we may assume that the vector bundles are trivial.
    Recall that a map to a compactly generated space is proper if and only if the preimages of compact subspaces of the target are (quasi)compact.  Therefore, we may assume that the map \(f\) is of the form \(K \times \mathcal{H}' \to K \times \mathcal{H}\) for \(K\) compact.
%
%    By \cite[Lemma 2.3]{BF1}, it is shown that the restrictions of \(f\) on the fibers are proper. So we only need to show that \(f\) is a closed map. 
    Since the preimage \(f^{-1}(y)\) of a point is contained in a bounded neighborhood (in a fiber) by assumption, it suffices to show that for any bounded subset \(A \subset \mathcal{H}'\) the restriction of \(f\) to \(K\times A\) is proper.

    Take any point \(x_0\in K\). Choose splittings \(\mathcal{H}' \cong \ker(l_{x_0})\times \im(l_{x_0})\) and \(\mathcal{H} \cong \im(l_{x_0}) \times \cok(l_{x_0})\). Replace \(K\) by a smaller compact neighborhood of \(x_0\) so that the composite \(\mathcal{H}' \xrightarrow{l_x} \mathcal{H} \xrightarrow{\pr} \im(l_{x_0})\) is surjective for every point \(x\in K\) (\cref{cor:SurjOpen}). So we take the kernels of that composite for \(x\in K\) and form a (finite-dimensional) vector bundle \(F\) over \(K\) (\cref{cor:submersiveKernel}). We may assume that the subbundle \(F\subset K \times \mathcal{H}'\) splits \(\rho \colon K\times \mathcal{H}' \to F\) (since it is finite-dimensional). Both \(c\) and \(\rho\) are compact maps. The map \(f\mid_{K\times A}\) decomposes into the following.
    \[
    \begin{tikzcd}[column sep = tiny, row sep=tiny]
        \!K {\times} A \! \ar[r] &  \mathcal{H} {\times} \overline{c(K{\times} A)} {\times}_K \overline{\rho(K{\times} A)} {\times}_K \overline{\rho l(K{\times} A)} \ar[r, "\cong"] &  (K{\times}\mathcal{H}) {\times}_K \overline{c(K{\times} A)} {\times}_K \overline{\rho(K{\times} A)} {\times}_K \overline{\rho l(K{\times} A)} \ar[r, "\pr"] & K{\times} \mathcal{H} \\
        (x,a) \ar[r, phantom, "\mapsto"] & 
        \begin{pmatrix}
            \pr_{\im(l_{x_0})}(l_x(a)) \\ c_x(a) \\ \rho_x(a) \\ \rho_x l_x (a)
        \end{pmatrix}, \qquad \qquad \qquad \begin{pmatrix}
            x \\ y \\ z \\ w
        \end{pmatrix} \ar[r, phantom, "\mapsto"] &
        \begin{pmatrix}
            x + z + w \\ y \\ z \\ w
        \end{pmatrix}
    \end{tikzcd}
    \] The first map is a closed embedding since it admits a retraction: \(a= \pr_{\im(l_{x_0})}(l_x(a))+ \rho_x (a)\). The third map is proper since \(c, \rho, \rho l\) are all compact maps.
\end{proof}

\section{The six-functor formalism for sheaves on topological spaces}\label{section:top6FF}

This section is aimed to recall the six-functor formalism for sheaves of spectra on topological spaces.
It serves as a spectral refinement of the work of~\cite{SS16}, where the functors \(f_!\dashv f^!\) are defined for separated locally proper maps.
One can compare to the paper~\cite{Volpe}, which provides the six-functor formalism for locally compact Hausdorff spaces.
Since we need the shriek functors along Fredholm maps between locally non-compact spaces (such as infinite-dimensional Banach spaces), we record the construction of that six-functor formalism, using the arguments in~\cite[A.5]{Mann} pioneered by Liu--Zheng.
We would like to note that it may also be possible to provide such a six-functor formalism by developing the theory of \emph{locally rigid algebras}~\cite{LocRig} in \(\Mod(\Pr^\L_\st)\) since separated locally proper maps of topological spaces are typical examples. We, however, do not pursue that perspective in this paper.

Let \(\Top\) denote the \((1,1)\)-category of topological spaces whereas \(\Top_\oo\) denote the category of \(\oo\)-topoi and geometric morphisms between them. We write \(\Top_\oo^{\L,\lex}\) for the subcategory of \(\Pr^\L\) equivalent to \(\Top_{\oo}^{\op}\).
Recall the functor
\[
(\Sh(-; \an), ({-})^\ast) \colon \Top^{\op} \to \Pr^\L
\]
from \cite[Ch.~6]{HTT} given as the following composite.
\[\begin{tikzcd}[column sep=large]
    \Top^\op \ar[r, "\Op(-)"] & \Locale^{\L,\lex} \simeq \Top_{\oo,0\textrm{-loc}}^{\L,\lex} \ar[r, hook] & \Top_\oo^{\L,\lex} \ar[r, hookrightarrow, "\mathrm{forgetful}"] & \Pr^\L
\end{tikzcd}\]
It takes a topological space \(Y\) to the category of sheaves in \(\an\), which is uniquely characterized by the following universal property
\begin{equation*}
    \Fun^\L(\Sh(Y;\an), \mathscr{C}) \xrightarrow[\sim]{\qquad} \Fun^{\mathrm{eff.epi}}(\Op(Y), \mathscr{C})
\end{equation*}
for \(\mathscr{C}\) a presentable category. (In fact this characterization proves the above functoriality.)
Here \(\Fun^\bullet(-,-)\) is the category of functors preserving colimit diagrams specified by \(\bullet\).

\begin{notation}\label{notation:Sh(Y)}
  For a topological space \(Y\), let \(\Sh(Y)\) denote the category \(\Sh(Y;\Sp)\) of sheaves of spectra, which is equivalent to \(\Fun^\R(\Sh(Y;\an)^\op, \Sp)\) and to \(\Sh(Y;\an)\otimes \Sp\).
  Thus, we have the \(\ast\)-functoriality:
  \begin{equation*}\begin{tikzcd}[row sep=0.3em, column sep=small]
  \Sh(-) \ar[r, phantom, ":"] & \Top^\op \ar[rr] && \CAlg(\Pr^\L_\st) \\
  & Y \ar[ddd, "f"'] \ar[rr, phantom, ""{name=Y}] && \Sh(Y) \\ &&& \\ &&& \\
  & X \ar[rr, phantom, ""'{name=X}] && \Sh(X) \ar[uuu, "f^\ast"']
  \arrow[from=Y, to=X, phantom, "\mapsto"]
  \end{tikzcd}\end{equation*}
  The right adjoint to the functor \(f^\ast\) is denoted by \(f_\ast\).
\end{notation}

In what follows, we will freely use the following fact.
\begin{remark}\label{fact:Radjoftensor}
  Let \(\mathscr{C}\), \(\mathscr{D}\) and \(\mathscr{E}\) be all presentable. Let \(f\colon \mathscr{C} \to \mathscr{D}\) be a colimit-preserving functor.

  If \(f\) admits a fully faithful right adjoint \(f^\R\), then the functor \((f\otimes\id_\mathscr{E})^\R\) right adjoint to the functor \(f\otimes \id_\mathscr{E}\colon \mathscr{C} \otimes \mathscr{E} \to \mathscr{D} \otimes \mathscr{E}\) is fully faithful.

  If the functor \(f^\R\) right adjoint to \(f\) admits a further right adjoint \(f^{\R\R}\), %and if the projection formula, 
  then the functor \((f\otimes \id)^\R\) is isomorphic to \(f^\R\otimes \id_\mathscr{E}\).
  Combined with the first remark, every functor \(f\colon \mathscr{C} \to \mathscr{D}\) that admits a colimit-preserving fully faithful right adjoint \(f^\R\) is tensored up to an adjoint triple \(f\otimes \id_\mathscr{E} \dashv f^\R\otimes \id_\mathscr{E} \dashv (f^\R \otimes \id_\mathscr{E})^\R\) with \(f^\R\otimes \id_\mathscr{E}\) fully faithful.
\end{remark}
The first fact follows from how a tensor product can be constructed in terms of presentations. The second fact follows from the observations that such an \(f\) is an internal left adjoint in \(\Pr^\L\) and that \((\Pr^\L, \otimes)\) has a structure of a symmetric monoidal \(2\)-category.

\subsection{Locally proper maps}

\begin{convention}\label{conv:(loc)proper}
  \(\mbox{}\)
  We here fix some basic terminology regarding compactness.
  \begin{enumerate}
      \item We say a map between topological spaces is \emph{proper} if it is separated and universally closed. For example, a map onto a singleton is proper if and only if the topological space is compact Hausdorff.

      \item We say a map \(f\) is \emph{locally universally closed} if it is locally proper in the sense of~\cite[Definition 2.3]{SS16}.
      That is, if for any (open) neighborhood \(V\) of a point \(y\) on the source there exist a (open) neighborhood \(U\) of \(f(y)\) in the target and a neighborhood \(N\subset f^{-1}(U)\cap V\) of \(y\) such that \(f\) restricts to a universally closed map \(N\to U\).
      For example, a map onto a singleton is locally universally closed if and only if the topological space is locally compact.

      \item We define a \emph{locally proper map} in the same way: For any neighborhood \(V\) of a point \(y\) on the source, there exist neighborhoods \(U\) of \(f(y)\) and \(N\subset f^{-1}(U)\cap V\) of \(y\) such that the restriction \(f\colon N \to U\) is proper.
      
      We remark that separated and locally universally closed maps are automatically separated locally proper.

      \item An \emph{immersion} is a map that induces a homeomorphism onto the image which is locally closed in the target. In other words, an immersion is a locally proper embedding.
  \end{enumerate}
\end{convention}

The following well-known lemma, the existence of a compactification relative to a base, will provide later (\cref{thm:top6FF}) well-defined shriek functors for such maps. One can refer to~\cite[II. 8]{FibwiseTop}, and we do not claim any originality of the following argument.

\begin{lemma}\label{lem:fiberwisecompactification}
    A map \(f\colon Y \to S\) is separated locally proper if and only if it is of the form \(f = pj\) where \(j\colon Y \to X\) is an open embedding and \(p\colon X \to S\) is a proper map.
\end{lemma}
\begin{proof}
    Since both \(p\) and \(j\) are separated locally proper and since the collection of separated locally proper maps are closed under compositions, any map of the form \(f=pj\) is separated locally proper.

    Conversely, we wish to construct an explicit fiberwise one-point compactification \(Y^{+S}\) of a separated locally proper map \(f\).
    Define a topology on the underlying set \(Y\sqcup S\) as follows.
    \begin{enumerate}
        \item For each open \(V\subset Y\), the subset \(V\sqcup \varnothing \subset Y \sqcup S\) is open.
        \item For each \(x=f(y)\in f(Y)\) and for each neighborhoods \(U\ni x\) and \(N\ni y\) with \(f\colon N \to U\) proper, the subset \(\left(f^{-1}(U) - N\right)\sqcup U\) is a neighborhood of \(x \in \varnothing\sqcup S \subset Y\sqcup S\).\label{item:nbhd/f(Y)}
        \item For each point \(x\notin f(Y)\) of \(S\) and for each neighborhood \(U\) of \(x\), the subset \(f^{-1}(U)\sqcup U\) is a neighborhood of \(x \in \varnothing\sqcup S \subset Y\sqcup S\).\label{item:nbhd/S-f(Y)}
        \item Define the topological space \(Y^{+S}\) to be the final topology on \(Y\sqcup S\) satisfying all the above conditions. Each of the conditions \eqref{item:nbhd/f(Y)} and \eqref{item:nbhd/S-f(Y)} forms a neighborhood basis of that point \(x\).
    \end{enumerate}
    Then the obvious maps \(j\colon Y \to Y^{+S}\) and \(p\colon Y^{+S} \to S\) are continuous, \(j\) is open, \(p\) is separated and \(p\) has (quasi)compact fibers.
    We need to prove that \(p\) is a closed map. In other words, we wish to show that for any point \(x\in S\), every neighborhood of \(p^{-1}(x)\) contains a subset of the form \(p^{-1}(U)\) for some neighborhood \(U\) of \(x\).

    First, for a point \(x\notin f(Y)\) of \(S\), any neighborhood of \(p^{-1}(x)\) contains a subset of the form \(f^{-1}(U)\sqcup U\) for some neighborhood \(U\ni x\) by \eqref{item:nbhd/S-f(Y)}, i.e., it contains \(p^{-1}(U)\) for some neighborhood \(U\) of \(x\).
    Next, for a point \(x\in f(Y)\), let \(V\) be a neighborhood of \(p^{-1}(x)\). Then since \(V\) is a neighborhood of \(x\), by \eqref{item:nbhd/f(Y)} it contains a subset of the form \((f^{-1}(U)-N)\sqcup U\) for some neighborhood \(U\) of \(x\) and some \(N\subset f^{-1}(U)\) with \(N \to U\) a closed map. Thus, we find a neighborhood \(U'\) of \(x\) contained in \(U\) such that \(f^{-1}(U')\cap N \subset V\cap N\). It follows that \(p^{-1}(U') \subset V\).
    Therefore, we conclude that \(p\) is a closed map.
\end{proof}

\begin{lemma}[Künneth equivalence] \label{lem:locproperpullback}
    Let \(f\) be a locally universally closed map and consider a cartesian square of topological spaces as follows. \begin{equation*}
        \begin{tikzcd}
            Y' \ar[r, "g'"] \ar[d, "f'"'] \ar[dr, "\lrcorner", phantom, very near start] & Y \ar[d, "f"] \\
            S' \ar[r, "g"] & S
        \end{tikzcd}
    \end{equation*}
    Then the following square is cocartesian in \(\Locale^{\L,\lex}\).
    \begin{equation*}
        \begin{tikzcd}
            \Op(S) \ar[r, "g^\ast"] \ar[d, "f^\ast"] & \Op(S') \ar[d] \\
            \Op(Y) \ar[r] & \Op(Y')
        \end{tikzcd}
    \end{equation*}

    Consequently, we have a cocartesian square in \(\CAlg(\Pr^\L_\st)\) which exhibits an equivalence 
    \[
    \Sh(S' \times_S Y) \simeq \Sh(S') \otimes_{\Sh(S)} \Sh(Y)
    \] by~\cite[Corollary 1.10]{aoki2023sheavesspectrumadjunction}.
\end{lemma}
\begin{proof}
    The proof given here is a slight variant of the argument given in \cite[Proposition 7.3.1.11]{HTT}. In fact, only the final step of our argument differs from the proof of \cite[7.3.1.11]{HTT}.
    
    Let \(\mathscr{L}\) denote the pushout \(\Op(S') \otimes_{\Op(S)}^{\lex} \Op(Y)\) in \(\Locale^{\L,\lex}\), which is equipped with the maps \(\psi_g\colon \Op(S') \to \mathscr{L}\), \(\psi_f\colon \Op(Y) \to \mathscr{L}\) exhibiting the cocartesian square. We let the element \(\psi_g(U')\cap \psi_f(V)\in \mathscr{L}\) be denoted by the symbol \(U'\otimes_S V\).
    Since \(\psi_g \circ g^\ast\) and \(\psi_f\circ f^\ast\) commute, we have \begin{equation*}
        %\label{eq:LocaleTensor}
        (U'\cap g^{-1}(U)) \otimes_S V = U' \otimes_S (f^{-1}(U)\cap V).
    \end{equation*}
    Since it is the pushout in \(\Locale^{\L,\lex}\), every element of \(\mathscr{L}\) is of the form \(\bigcup_{\alpha} U'_\alpha \otimes_S V_\alpha\).

    Let \(\phi \colon \mathscr{L} \to \Op(S'\times_S Y)\) be the map in \(\Locale^{\L,\lex}\) provided by the universal property of \(\mathscr{L}\).
    We define a (left-exact but not necessarily colimit-preserving) functor \(\theta \colon \Op(S'\times_S Y) \to \mathscr{L}\) by the formula \[
    \theta(W) = \bigcup_{U'\times_S V \subset W} U' \otimes_S V.
    \] We can easily see that \(\phi \theta =\id\).
    Therefore, it suffices to show that \(\theta\phi\) is also the identity. That is, we wish to show that \[
    \bigcup_{U'\times_S V\subset \bigcup_\alpha U'_\alpha \times_S V_\alpha} U' \otimes_S V = \bigcup_{\alpha} U'_\alpha \otimes_S V_\alpha. 
    \] Here the right hand side is \(\le\) the left hand side by trivial reason. So we claim the converse inequality.

    Since \(f\colon Y \to S\) is locally universally closed, any open subset \(V\) can be written as \(\bigcup_\beta T_\beta^\intr\) where \(T_\beta\subset V\) and the map \(f\) restricts to a universally closed map \(T_\beta \to U_\beta\) for some open subset \(U_\beta\subset S\).
    Since \(U'\otimes_S V = \bigcup_\beta U'\otimes_S T_\beta^\intr\), it suffices to prove that \(U'\otimes_S T_\beta^\intr \subset \bigcup_\alpha U'_\alpha \otimes_S V_\alpha\) whenever \(U'\times_S T_\beta^\intr \subset \bigcup_\alpha U'_\alpha \times_S V_\alpha\).

    Fix \(\beta\).
    We may replace \(U'\) by \(U'\cap g^{-1}(U_\beta)\) because %\((U'\cap g^{-1}(U_\beta)) \times_S T_\beta^\intr = U' \times_S T_\beta^\intr\) and 
    \((U'\cap g^{-1}(U_\beta))\otimes_S T_\beta^\intr = U' \otimes_S (f^{-1}(U_\beta)\cap T_\beta^\intr)=U'\otimes_S T_\beta^\intr\). So we may assume that \(U'\subset g^{-1}(U_\beta)\).
    
    Fix a point \(x'\in U'\) and set \(x= g(x') \in U_\beta\).
    Since \(\{x'\}\times_S T_\beta = \{x'\} \times (f^{-1}(x)\cap T_\beta)\) is quasicompact, there exist finitely many \(\alpha_1,\dots,\alpha_n\) such that \(x'\in \bigcap_{i=1}^n U'_{\alpha_i}\) and \[
    f^{-1}(x) \cap T_\beta \subset V_{\alpha_1} \cup \cdots \cup V_{\alpha_n}.
    \] Since the map \(T_\beta \to U_\beta\) is closed, we can find an open neighborhood \(U_0\) of \(x\) in \(U\) such that \[
    f^{-1}(U_0) \cap T_\beta \subset V_{\alpha_1} \cup \cdots \cup V_{\alpha_n}.
    \] Let \(U'(x')\) denote the open neighborhood \(g^{-1}(U_0) \cap \bigcap_{i=1}^n U'_{\alpha_i}\) of \(x'\).
    Then we have \begin{align*}\textstyle
        U'(x') \otimes_S T_\beta^\intr &= \left(\bigcap U'_{\alpha_i}\right) \otimes_S \left(f^{-1}(U_0) \cap T_\beta^\intr\right) \\
        &\le \left(\bigcap U'_{\alpha_i}\right) \otimes_S \bigcup_{i} V_{\alpha_i} \\
        & = \bigcup_i \left({\textstyle\bigcap_j U'_{\alpha_j}}\right) \otimes_S V_{\alpha_i} \le \bigcup_i U'_{\alpha_i} \otimes_S V_{\alpha_i} \; \le \; \bigcup_\alpha U'_\alpha \otimes_S V_\alpha.
        %\\ &\le \bigcup_i U'_{\alpha_i} \otimes_S V_{\alpha_i} \; \le \; \bigcup_\alpha U'_\alpha \otimes_S V_\alpha.
    \end{align*} Taking the colimit in \({x'\in U'}\), we get \(U' \otimes_S T_\beta^\intr \le \bigcup_\alpha U'_\alpha \otimes_S V_\alpha\) as desired.
\end{proof}

\subsection{Open/closed immersions}

\begin{theorem}[Open-closed recollement] \label{thm:OpenClosed}
    Let \(j\colon U \hookrightarrow Y\) be an open subspace and \(i\colon Z \hookrightarrow Y\) the closed complement. Then:
    \begin{enumerate}
        \item The functor \(j^\ast \colon \Sh(Y;\an) \to \Sh(U;\an)\) admits a left adjoint \(j_!\), which is fully faithful and satisfies the projection formula, i.e., the map \[j_!((-)\times_{j^\ast(-)} j^\ast(-)) \to j_!(-)\times_{(-)} (-)\] is an isomorphism.

        In particular, the functor \(j^\ast \colon \Sh(Y) \to \Sh(U)\) admits a fully faithful left adjoint whose oplax \(\Sh(Y)\)-linear structure is strongly \(\Sh(Y)\)-linear.

        \item The functor \(i_\ast \colon \Sh(Z;\an) \to \Sh(Y;\an)\) is fully faithful, preserves filtered colimits. %and the canonical lax \(\Sh(Y;\an)\)-linear structure is strongly \(\Sh(Y;\an)\)-linear.
        The functor \(i^\ast \colon \Sh(Y;\an) \to \Sh(Z;\an)\) exhibits \(\Sh(Z;\an)\) as a localization of \(\Sh(Y;\an)\) with respect to the maps of the form \(\varnothing \to j_!(-)\).
        
        In particular, the functor \(i_\ast \colon \Sh(Z) \hookrightarrow \Sh(Y)\), which is defined as the right adjoint to \(i^\ast\), admits a right adjoint \(i^! \colon \Sh(Y) \to \Sh(Z)\).

        \item In \(\Sh(Y)\), we have the following fiber sequences. \begin{gather*}
        j_! j^\ast \to \id \to i_\ast i^\ast \\
        i_\ast i^! \to \id \to j_\ast j^\ast
        \end{gather*} (In fact, each of the fiber sequences can be obtained as the adjoints to the other fiber sequence.)
        In other words, the null sequence %\[\Sh(Z) \xrightarrow{i_\ast} \Sh(Y) \xrightarrow{j^\ast} \Sh(U)\]
        \[\Sh(U) \xrightarrow{j_!} \Sh(Y) \xrightarrow{i^\ast} \Sh(Z)\]
        is a bifiber sequence in \(\Pr^\L_\st\).
    \end{enumerate}
\end{theorem}
\begin{proof}
    The unstable claims are the combination of~\cite[§6.3.5, §7.3.2]{HTT}.
    The other stable claims are then formal consequences since \(\Sh(-) = \Sh(-;\an)\otimes \Sp\).
\end{proof}

From now on, we will always consider the categories of sheaves of spectra.
In particular, the functors such as \(f^\ast, f_\ast, j_!, i^!\) will always take values in the categories of sheaves of spectra.

\begin{proposition}
    [Smooth basechange] \label{prop:SmBC}
    Let \begin{equation*}
        \begin{tikzcd}
            V \ar[r, "j'"] \ar[d, "f'"'] \ar[dr, "\lrcorner", phantom, very near start] & Y \ar[d, "f"] \\
            U \ar[r, "j"] & S
        \end{tikzcd}
    \end{equation*} be a cartesian square of topological spaces in which \(j\) is an open embedding.
    Then the map (in \(\Sh(U)\)) \[
    j'_! f'^\ast \to f^\ast j_!
    \] is an isomorphism.
\end{proposition}
\begin{proof}
This is~\cite[Lemma 3.25]{Volpe}.
One can also prove it using the unstable version~\cite[Remark 6.3.5.8]{HTT} together with the observation that both sides are left adjoint functors.
\end{proof}

\subsection{Proper pushforward}

\begin{theorem}
    [Proper basechange] \label{thm:PBC}
    Let \begin{equation*}
        \begin{tikzcd}
            Y' \ar[r, "g'"] \ar[d, "p'"'] \ar[dr, "\lrcorner", phantom, very near start] & Y \ar[d, "p"] \\
            S' \ar[r, "g"] & S
        \end{tikzcd}
    \end{equation*} be a cartesian diagram of topological spaces in which \(p\) is proper.
    Then the map \[g^\ast p_\ast \to p'_\ast g'^\ast\] is an isomorphism. Moreover, the functor \(p_\ast\) preserves colimits and thus admits a right adjoint \(p^!\).
\end{theorem}
\begin{proof}
    Combine \cite{Proper} and \cite[Proposition 3.8]{NonabPBC}.
    We note that although the notion of proper morphisms of \(\oo\)-topoi is defined in terms of pullbacks in \(\Top_\oo\), the pullbacks of \(0\)-localic \(\oo\)-topoi are given as the pushouts in \(\Locale^{\L,\lex}\), so that we have the correct cartesian square of the associated \(\oo\)-topoi due to \cref{lem:locproperpullback}:  \qedhere
    \begin{equation*}
        \begin{tikzcd}
            \Sh(Y';\an) \ar[r, "g'_\ast"] \ar[d, "p'_\ast"'] \ar[dr, "\lrcorner", phantom, very near start] & \Sh(Y;\an) \ar[d, "p_\ast"] & \mbox{} \ar[d, phantom, "\in \Top_\oo"] \\
            \Sh(S';\an) \ar[r, "g_\ast"] & \Sh(S;\an) & \mbox{}
        \end{tikzcd}
    \end{equation*}
\end{proof}

\begin{lemma}\label{lem:jshriekp'star}
    Let \begin{equation*}
        \begin{tikzcd}
            U' \ar[r, "j'"] \ar[d, "p'"'] \ar[dr, "\lrcorner", phantom, very near start] & Y \ar[d, "p"] \\
            U \ar[r, "j"] & S
        \end{tikzcd}
    \end{equation*} be a cartesian square where \(j\) is an open embedding and \(p\) is proper.
    Then the map \[
    j_! p'_\ast \to p_\ast j'_!
    \] is an isomorphism
\end{lemma}
\begin{proof}
    Let \(i\colon Z \to S\) be the closed complement to the map \(j\).
    Since \(j^\ast\) and \(i^\ast\) are jointly conservative, it suffices to show that the two maps
    \begin{align*}
    i^\ast j_! p'_\ast \to i^\ast p_\ast j'_! && p'_\ast \to j^\ast p_\ast j'_!
    \end{align*}
    are isomorphisms.
    First, using the proper basechange and observing that \(j^{-1}(Z)=\varnothing\), we immediately see that the both sides of the first map are zero.
    Next, under the proper basechange \(j^\ast p_\ast \simeq p'_\ast j'^\ast\), the second map is equivalent to the unit map \(p'_\ast \to p'_\ast j'^\ast j'_!\), which is an isomorphism since \(j'_!\) is fully faithful.
\end{proof}

\begin{lemma}
    [Proper projection formula] \label{lem:ProperProjFormula}
    Let \(p\colon Y \to S\) be a proper map.
    Then the map \[
    p_\ast(-) \otimes (-) \to p_\ast(-\otimes p^\ast(-))
    \] is an isomorphism.
\end{lemma}
\begin{proof}
    Since both sides commute with colimits and finite limits in each variable, and since \(\Sh(S)\) is generated under colimits and finite limits by the objects of the form \(j_!(\unit)\) for \(j\colon U\hookrightarrow S\) open, it suffices to show that the map
    \[
    p_\ast(-)\otimes j_!(\unit) \to p_\ast(- \otimes p^\ast j_!(\unit))
    \]
    is an isomorphism. Consider the cartesian square \(\begin{tikzcd}[row sep=small, column sep=small]
        p^{-1}U \ar[r, "j'"] \ar[d, "p'"'] & Y \ar[d, "p"] \\
        U \ar[r, "j"'] & S
    \end{tikzcd}\). Then under the smooth basechange \(p^\ast j_! \simeq j'_! p'^\ast\), the smooth projection formula for \(j'_!\), the isomorphism \(p_\ast j'_! \simeq j_! p'_\ast\) in \cref{lem:jshriekp'star} and the proper basechange \(p'_\ast j'^\ast \simeq j^\ast p_\ast\), the map in question takes the form \[
    p_\ast(-) \otimes j_!(\unit) \to j_! j^\ast p_\ast,
    \] which is exactly the isomorphism given by the smooth projection formula for \(j\).
\end{proof}

\subsection{Shriek functors}

\begin{theorem}\label{thm:top6FF}
    Let \(E\) denote the class of separated locally proper maps.
    Then the functor \[(\Sh(-),(-)^\ast) \colon \Top^\op \to \CAlg(\Pr^\L_\st)\] extends to a six-functor formalism, i.e., a lax symmetric monoidal functor\footnote{Here we used the notations from~\cite[A.5.2, A.5.4]{Mann}.} of the form
    \begin{equation*}
        \Sh(-) \colon \Corr(\Top)_{E,\textit{all}} \longrightarrow \Pr^\L_\st
    \end{equation*}
    such that for any locally proper map \(f\) together with a decomposition \(f=pj\) for \(j\) an open embedding and \(p\) a proper map, the functor \(f_!\) is identified with \(p_\ast j_!\).
\end{theorem}

The right adjoint to a functor of the form \(f_!\) will be denoted by \(f^!\).

\begin{proof}
    We provide such a functor by means of~\cite[Proposition A.5.10]{Mann}.
    Combining \cref{lem:fiberwisecompactification}, \cref{thm:OpenClosed}, \cref{prop:SmBC}, \cref{thm:PBC}, \cref{lem:jshriekp'star} and \cref{lem:ProperProjFormula}, we see the assumptions of loc.~cit.~are all satisfied. This completes the proof.
\end{proof}

\begin{comment}
\begin{remark}
    The recent development of the theory of \emph{locally rigid algebras} in \(\Mod_\mathscr{V}(\Pr^\L_\st)\)~\cite{LocRig} may also provide such a six-functor formalism possibly in a more canonical way.
\end{remark}
\end{comment}

\begin{corollary}
    [{\cite[Proposition A.5.8]{Mann}}]
    For a separated locally proper map \(f\), we have the projection formula: \[
    f_! \text{ is } f^\ast\text{-linear,}
    \] the basechange theorem \[
    g^\ast f_! \simeq f'_! g'^\ast,
    \] and the functoriality \[
    f_1 {}_! f_0{}_! \simeq (f_1 f_0)_!.
    \]
\end{corollary}

\subsection{Homotopy and Monodromy}

\begin{theorem}[\(\AA^1\)-invariance]
    \label{thm:A1-inv}
    For any topological space \(S\),
    the map \(\pr\colon S \times [0,1] \to S\) induces a fully faithful functor
    \[\Sh(S) \xrightarrow[\pr^\ast]{} \Sh(S\times [0,1]).\]
\end{theorem}
\begin{proof}
    This is a variant of~\cite[Lemma A.2.9]{HA}.
    Although it proves the unstable version, since the functor \(\pr^\ast \colon \Sh(S\times[0,1];\an) \to \Sh(S;\an)\) admits a left adjoint \(\pr_\sharp\), it follows that \(\pr^\ast \colon \Sh(S)\to \Sh(S\times[0,1])\) is fully faithful as well.
\end{proof}	

\begin{corollary}
    Let two map \(f\) and \(g\) from \(p\) to \(q\) be homotopic over the base \(S\) as in the following diagram.
    \begin{equation*}
        \begin{tikzcd}
            & X \ar[dl, "\id"', end anchor=north east] \ar[d, "i_0"] \ar[rrd, "f", end anchor=north west] && \\
            X \ar[ddrr, "p"', bend right, start anchor=south] & \ar[l, "\pr"{description}] X\times [0,1] \ar[rr, "h"{description}] && Y \ar[ddl, "q", bend left] \\
            & X \ar[ul, "\id"{description}, end anchor=south east] \ar[u, "i_1"'] \ar[rru, "g"', end anchor=south west] && \\
            %&&&\\
            && S &
        \end{tikzcd}
    \end{equation*}
    Then we have the following diagram consisting of adjunction units.
    \[\begin{tikzcd}[row sep = small]
        & p_\ast\pr_\ast i_0{}_\ast i_0^\ast \pr^\ast p^\ast = q_\ast f_\ast f^\ast q^\ast & \\
        p_\ast p^\ast \ar[ru, equal, end anchor = west] \ar[rd, equal, end anchor=west] \ar[r, "\sim"] & p_\ast \pr_\ast \pr^\ast p^\ast \ar[u, shift left=1.3em] \ar[d, shift right = 1.3em] & \ar[l] q_\ast q^\ast \ar[lu, end anchor = south east] \ar[ld, end anchor = north east] \\
        & p_\ast \pr_\ast i_1{}_\ast i_1^\ast \pr^\ast p^\ast = q_\ast g_\ast g^\ast q^\ast &
    \end{tikzcd}\]
    The diagram says that the two maps \(f^\ast\) and \(g^\ast\) on the cohomology (relative to \(S\)) \[f^\ast , g^\ast \colon \H^{-\bullet}(Y/S;-) \rightrightarrows \H^{-\bullet}(X/S;-)\] coincide (witnessed by the above homotopy diagram).
\end{corollary}

%An immediate consequence of the homotopy invariance is the following.

\begin{definition}\label{dfn:contractible}
    We say a map \(q\colon Y \to S\) is \emph{contractible} if there exist a section \(e \colon S \to Y\) and a homotopy \(eq\sim \id_Y\) over \(S\).
\end{definition}

Note that if \(q\) is contractible, then any section \(s\) will also serve as \(e\) in the above definition.

\begin{corollary}\label{cor:contractible}
    Let \(q\colon Y \to S\) be a contractible map.
    Then the functor \(q^\ast\) is fully faithful.
\end{corollary}

Thus, for a contractible morphism \(q\), we have an isomorphism \(q_\ast q^\ast \simeq s^\ast q^\ast\). In the following, we provide a sufficient condition for a sheaf to be in the image of \(q^\ast\).

\begin{definition}
    Let \(q\colon Y \to S\).
    Let us say an object \(A\in \Sh(Y)\) is \emph{constant along \(q\)} if it belongs to the essential image of \(q^\ast\).
    Similarly, \(A\) is said to be \emph{locally constant along \(q\)} if there exists an open cover \(j\colon \coprod U_\alpha \to Y\) such that each \(j_\alpha^\ast(A)\in \Sh(U_\alpha)\) is constant along \(qj_\alpha \colon U_\alpha \to S\).
\end{definition}

According to the definition, we obtain two full subcategories \[\Sh_{\cnst/S}(Y)\subset \Sh_{\lcnst/S}(Y)\subset \Sh(Y)\] spanned by constant and locally constant objects along \(q\), respectively.

\begin{lemma}\label{lem:globallyconstant}
    Suppose that \(q\colon Y \to S\) is a Banach ({or more generally locally convex}) vector bundle.
    Let \(A \in \Sh(Y)\) be locally constant along \(q\). Then it is constant along \(q\).

    The same holds if \(q\) is a composite \(Y \hookrightarrow E \to S\) where \(E \to S\) is a Banach vector bundle and \(Y \subset E\) is %a convex subspace.
    locally trivialized as a product of an open in \(S\) and a convex subspace of a fiber of \(E\to S\). For example, it works for \(q = \pr \colon S\times [0,1] \to S\).
\end{lemma}
\begin{proof}
    First, we observe that in the following commutative diagram with \(j\) an open embedding
    \begin{equation*}
        \begin{tikzcd}
            V \ar[r] \ar[d, "p"'] \ar[rd, "q_V"{description}] & Y \ar[d, "q"] \\
            U \ar[r, "j"', hookrightarrow] & S
        \end{tikzcd}
    \end{equation*}
    we have
    \(q_V^\ast = p^\ast j^\ast\) and also \(p^\ast \simeq p^\ast j^\ast j_\ast=q_V^\ast j_\ast\) since \(j_\ast\) is fully faithful. Thus, we have an equality \(\Sh_{\lcnst/S}(V) = \Sh_{\lcnst/U}(V)\) in this situation. Note that since \(q\) is an open map, the above \(U\) can be \(q(V)\).
    
    Consider the sheaves\footnote{Follows from \emph{descent}~\cite[Theorem 6.1.3.9]{HTT} for \(\oo\)-topoi.} \begin{align*}
    \Sh(-) &\colon \Op(Y)^\op \to \Cat \\
    \Sh(-) &\colon \Op(S)^\op \to \Cat.
    \end{align*}
    The subfunctor \[
    \Sh_{\lcnst/S}(-) \colon \Op(Y)^\op \to \Cat
    \] is also a sheaf by definition.
    Let \(\mathscr{C}(-)\colon \Op(Y)^\op \to \Cat\) denote the 
    sheaf obtained by applying the functor \(q^\ast \colon \Sh(S;\Cat)\to \Sh(Y;\Cat)\) to the sheaf \(\Sh(-)\colon \Op(S)^\op \to \Cat\).
    We have a comparison map \(\mathscr{C}(-) \to \Sh_{\lcnst/S}(-)\) by adjunction \(q^\ast\dashv q_\ast\).
    For an open subset \(V\subset Y\), \(\mathscr{C}(V)\) is given by \(q_{V\ast}q_V^\ast(\Sh(-))(q(V))\), which is equivalent to \(\Sh(q(V))\) if \(V \to q(V)\) is contractible (by the unstable homotopy invariance). 
    
    We claim that the map \(\mathscr{C}(-) \to \Sh_{\lcnst/S}(-)\) is an equivalence, which in particular implies that \[q^\ast \colon \Sh(S) \simeq \mathscr{C}(Y) \to \Sh_{\lcnst/S}(Y)\] is an equivalence.
    By descent, it suffices to show that \(\mathscr{C}(q^{-1}U) \to \Sh_{\lcnst/S}(q^{-1}U) = \Sh_{\lcnst/U}(q^{-1}U)\) is an equivalence for \(q^{-1}U \cong U\times \mathcal{H}\).
    Therefore, we may assume that \(q\colon Y \to S\) is trivial: \(S \times \mathcal{H} \to S\).

     Let \(M\in \Sh(S\times \mathcal{H})\) be locally constant over \(S\).
     We take an open cover \(\{V_i\}\) of \(S\times \mathcal{H}\) of the form \(V_i = U_i \times B_i\) for \(B_i\) convex open in \(\mathcal{H}\) such that \(M\) is pulled-back to be constant over \(\coprod V_i\).
     Each of the functor \(\Sh(U_{i_0}\cap\cdots\cap U_{i_n}) \to \Sh(V_{i_0}\cap\cdots\cap V_{i_n})\) is fully faithful since \(V_{i_0,\dots,i_n} \to U_{i_0,\dots,i_n}\) is contractible. The fully faithful embedding 
    \begin{align*}
         \mathscr{C}(S\times \mathcal{H}) & \xrightarrow{\sim} \lim_{[n]\in\DDelta} \prod_{i_0,\dots,i_n} \mathscr{C}(V_{i_0}\cap\cdots\cap V_{i_n}) \\
         &\xleftarrow{\sim} \lim_{\DDelta} \prod \Sh(U_{i_0}\cap\cdots\cap U_{i_n}) \hookrightarrow \lim_{\DDelta} \prod_{i_0\dots,i_n} \Sh_{\lcnst/S}(V_{i_0}\cap\cdots\cap V_{i_n}) \xleftarrow{\sim} \Sh_{\lcnst/S}(S\times \mathcal{H})
     \end{align*}
     contains \(M\) in its image by construction.
     Therefore, the map \(\mathscr{C}(S\times \mathcal{H}) \to \Sh_{\lcnst/S}(S\times\mathcal{H})\) is fully faithful and essentially surjective.
\end{proof}

\begin{corollary}\label{lem:VBinvariance}
    Let \(q\colon Y \to S\) be as in \cref{lem:globallyconstant} and \(s\colon S \to Y\) be a section.
    Then we have an isomorphism \[
    q_\ast \simeq s^\ast
    \] on \(\Sh_{\lcnst/S}(Y)\) the locally constant sheaves relative to \(S\).
\end{corollary}
\begin{proof}
    Consider the map \(%
    q_\ast = s^\ast q^\ast q_\ast \to s^\ast.%
    \)
    It is an isomorphism after precomposed with \(q^\ast\) as we have already seen.
    This suffices for the proof.
\end{proof}

We record the monodromy equivalence theorem, which tells us that the category of sheaves contains local systems as its full subcategory.
We do not use this theorem in the paper, but we would like to regard it as a motivation or justification to use sheaves of spectra.

\begin{theorem}
    [{\cite[A.1.15]{HA}}]
    \label{thm:monodromy}
    Let \(\LocSys\colon \Pro(\an) \to \Top_\oo^{\R}\) be the right Kan extension of the functor \((\PSh(-), (-)_\ast) \colon \an \to \Top_\oo\). Since it admits a left adjoint~\cite[7.1.6.15]{HTT}, we let \(\Pi_\oo\) denote the left adjoint \(\Top_\oo \to \Pro(\an)\).
    
    Let \(S\) be a locally contractible and hypercomplete topological space.
    Then:
    \begin{enumerate}
        \item \(\Pi_\oo S\) is pro-constant. It is equivalent to the singular simplicial set \(\Sing(S)\) if, for example, \(S\) is homotopy equivalent to a CW-complex\footnote{More generally, \(\Pi_\oo S \simeq \Sing (S)\) if \(S\) is paracompact, locally contractible and hypercomplete.}.

        \item The functor \(\Sh(S;\an) \to \Fun(\Pi_\oo S, \an)\) induces an equivalence \[\Sh_{\lcnst}(S;\an) \xrightarrow{\sim} \Fun(\Pi_\oo S, \an)\] where the left hand side is the full subcategory spanned by the locally constant objects.
    \end{enumerate}
\end{theorem}

\subsection{Duality results}

\begin{notation}\label{conv:smoothness}
    \(\mbox{}\)We would like to use some ad hoc terminology concerning cohomological smoothness.
    \begin{enumerate}
        \item A map \(p\colon Y \to S\) is said to be \emph{topologically smooth} if it is an open map and for every point \(y\in Y\) there exists a neighborhood \(V\) of \(y\) and there exists a homeomorphism \(V \cong p(V) \times \RR^d\) such that the map \(p|_V\) agrees with the composite \(V \cong p(V) \times \RR^d \xrightarrow{\pr} p(V)\).

        \item By a \emph{topological regular immersion}, we mean a closed embedding \(i\colon Z \hookrightarrow Y\) such that %for every point \(x\in Z\) there exists a neighborhood \(V\) of \(x\) in \(Y\) and there exists a homeomorphism \(V \cong (Z\cap V) \times \RR^c\) such that the map \(i'\colon Z\cap V \to V\) agrees with the zero section \(Z\cap V \to (Z\cap V)\times \RR^c \cong V\).
        there exists %a retraction \( r\colon N \to Z \) defined on some open neighborhood \( N \) of \( i \) such that \( r \) is homeomorphic to a vector bundle over \( Z \).
        an open neighborhood \( N \) of \( i \) such that the map \( Z \to N \) is homeomorphic to the zero section \( Z \to E \) for some vector bundle \( E \) over \( Z \).
    \end{enumerate}
\end{notation}

\begin{construction}\label{const:Th(microbundle)}
  We define Thom spectrum sheaves for (relative tangent) microbundles, which is a variant of Thom spectrum sheaves for vector bundles \cref{const:Th(ind(l))}.

  Let \( p\colon Y \to S \) be topologically smooth.
  Then define \( \Th(T_p) \) to be \( \pi_1{}_\sharp \Delta_!(\unit) \) where \( \Delta \colon Y \to Y \times_S Y \) is the diagonal, \( \pi_1 \colon Y\times_S Y \to Y \) is the first projection, and \( \pi_1{}_\sharp \) is left adjoint to \( \pi_1^\ast \), which exists by the argument of \cref{prop:PD}.
\end{construction}

\begin{remark}\label{rem:purityTh}
To justify our notation, let \( p \colon Y \to S \) be
a \( C^1 \)-submersion between, say, \( C^1 \)-manifolds.
Then we have the relative tangent vector bundle \(q\colon T_p \to Y\) with zero section \( e\colon Y \to T_p \), and there exists a tubular neighborhood of the diagonal \( \Delta \), i.e., an open embedding \( j \colon T_p \to Y \times_S Y \) such that the diagram
\[\begin{tikzcd}
  Y \ar[r, "\Delta"] \ar[d, "e"']	& Y \times_S Y \ar[d, "\pi_1"] \\
  T_p \ar[r, "q"'] \ar[ru, hookrightarrow, "j"]	& Y
\end{tikzcd}\]
commutes.
Therefore, we have an isomorphism \( \pi_1{}_\sharp \Delta_!(\unit) \simeq q_\sharp e_!(\unit) \) since \( j_! = j_\sharp \).
Furthermore, the latter \( q_\sharp e_!(\unit) \) is also of the form \( q_!(q^!(\unit) \otimes e_!(\unit)) \) by \cref{prop:PD}, the projection formula implies \( q_!(q^!(\unit)\otimes e_!(\unit)) \simeq q_!e_! e^\ast q^!(\unit) = e^\ast q^!(\unit) \), and \( e^\ast q^!(\unit) \simeq q_\ast q^!(\unit) \) by \cref{lem:VBinvariance}.
\end{remark}

\begin{remark}\label{rem:KunnethSauve}
    Let \(Y\) (and \(Y'\)) be a locally compact space so that we have the Künneth equivalence \(\Sh(S\times Y) \xleftarrow{\sim} \Sh(S) \otimes \Sh(Y)\).
    By the projection formula, it follows that the functor \[(f\times g)_! \colon \Sh(S'\times Y') \to \Sh(S\times Y)\] is identified with \(f_!\otimes g_!\) through the Künneth equivalence.

    Therefore, whenever we have an isomorphism of the form \(g^!(\unit) \otimes g^\ast (-) \simeq g^!(-)\), which in particular implies that \(g^!\) admits a further right adjoint, we identify the functor \((\id_S \times g)^!\) with \(\id_{\Sh(S)} \otimes g^!\). In other words, \((\id_S \times g)^! \simeq \pr_2^\ast g^!\) where \(\pr_2\) is the projection \(S\times Y' \to Y'\).
\end{remark}

\begin{proposition}[Poincaré duality]\label{prop:PD}
    A topologically smooth %\footnote{See \cref{conv:smoothness}.} 
    morphism \(p\) is cohomologically smooth in the following sense:
    \(p\) is \(!\)-able, 
    \(p^!(\unit)\) is locally constant, \(\otimes\)-invertible, the map\footnote{\(\begin{tikzcd}[ampersand replacement=\&] \cdot \ar[r, "\pi_1"] \ar[d, "\pi_2"'] \ar[rd, phantom, very near start, "\lrcorner"] \& \cdot \ar[d, "p"] \\ \cdot \ar[r, "p"'] \& \cdot \end{tikzcd}\)} \(\pi_1^\ast p^!(\unit) \to \pi_2^!(\unit)\) is an isomorphism and the map \[
    p^!(\unit) \otimes p^\ast(-) \to p^!(-)
    \] is an isomorphism.
    Moreover, \( p^!(\unit) \) is identified with the Thom spectrum sheaf \( \Th(T_p) \) of the tangent microbundle (\cref{const:Th(microbundle)}).
\end{proposition}

\begin{proposition}[Relative purity]\label{prop:relPurity}
    Let \(i\colon Z \to Y\) be a topological regular %\footnote{See \cref{conv:smoothness}.} 
    immersion.
    Then \(i^!(\unit)\) is locally constant, \(\otimes\)-invertible and is given as \(\Th(-N_i)\) the negative Thom spectrum sheaf of the normal sheaf, i.e., the \( \otimes \)-inverse of \( \Th(N_i) \).
\end{proposition}

\begin{proof}[\cref{prop:PD} and \cref{prop:relPurity}]
    Both proofs can essentially be found in~\cite{Volpe}.

    In fact,
    to see that the maps \( \pi_1^\ast p^! (\unit) \to \pi_2^!(\unit) \) and \( p^!(\unit)\otimes p^\ast(-)\to p^!(-) \) are isomorphisms, by descent it suffices to show for the trivial vector bundle \( \RR^n \times S \to S \), which in turn follows by \cref{rem:KunnethSauve} from the case for the map \( \RR^n \to \ast \). This case is already proved in~\cite{Volpe}.
    %%By \cref{rem:KunnethSauve}, the claim for the case \(\RR^n \times S \to S\) follows.
%
    Moreover, the isomorphism \( p^!(\unit) \simeq \Th(T_p) \) can be obtained as follows.
    By \( \pi_1^!(\unit) \otimes \pi_1^\ast \simeq \pi_1^! \), we have \( \pi_1{}_\sharp = \pi_1{}_!(\pi_1^!(\unit) \otimes -) \).
    By the projection formula, \( \pi_1{}_!(\pi_1^!(\unit) \otimes \Delta_!(\unit) \simeq \Delta^\ast \pi_1^!(\unit) \).
    Finally, by \( \pi_1^!(\unit) \simeq \pi_2^\ast p^!(\unit) \), we have \( \pi_1{}_\sharp \Delta_!(\unit) \simeq p^!(\unit) \).

    For relative purity,
    let \( i\colon Z \to X \) factor as the zero-section \( e\colon Z \to N \) of some vector bundle \( p\colon N \to Z \), followed by an open embedding \( j\colon N \to X \).
    Then \( i^!(\unit) = e^! j^!(\unit) = e^!(\unit) \).
    By \cref{lem:PicardUpperShriek}, we have \( e^!(\unit) \otimes e^\ast r^!(\unit) \simeq e^! p^!(\unit) = \unit \). Thus, \( e^!(\unit) \simeq i^!(\unit) \) is \( \otimes \)-invertible with inverse \( e^\ast p^!(\unit) \), which is isomorphic to \( \Th(N) \) (see \cref{rem:purityTh}).
\end{proof}

\begin{lemma}\label{lem:PicardUpperShriek}
    Let \(f\colon Y \to S\) be any map.
    Then the map \[f^!(\unit)\otimes f^\ast(-) \to f^!(-)\]
    is an isomorphism on \(\Pic(\Sh(S))\) the \(\otimes\)-invertible objects.
\end{lemma}
\begin{proof}
    We first note that a \(\otimes\)-invertible object \(E\in \Sh(S)\) is dualizable and the inverse is identified with the dual. Also, the symmetric monoidal functor \(f^\ast\) preserves duality data. In particular, the functors \((-)\otimes E\) and \((-)\otimes f^\ast E\) admit left adjoints \((-)\otimes E^{-1}\) and \((-)\otimes f^\ast E^{-1}\), respectively.
    Passing to the left adjoints of the both sides of the map
    \[
    f^!(-) \otimes f^\ast(E) \to f^!(-\otimes E)
    \]
    we get the map
    \[
    f_!( -\otimes f^\ast E^{-1} ) \leftarrow f_!(-) \otimes E^{-1},
    \] which is the isomorphism exhibiting the projection formula.
\end{proof}

\begin{lemma}
    \label{lem:basecahgeofdualizingsheaves}
    Let \(p\colon X \to S\) be topologically smooth. Then for any \(g\colon S' \to S\), the map \(g'^\ast p^!(\unit) \to p'^!(\unit)\), induced from the cartesian square \(\begin{tikzcd}[row sep=small, column sep=small]
        X' \ar[r, "g'"] \ar[d, "p'"'] \ar[rd, phantom, very near start, "\lrcorner"] & X \ar[d, "p"] \\
        S' \ar[r, "g"'] & S
    \end{tikzcd}\), is an isomorphism.

    Let \(e\colon S \to E\) be the zero section to a vector bundle \(p\colon E \to S\).
    For any \(g\colon S' \to S\), consider the cartesian rectangle \[
    \begin{tikzcd}
        S' \ar[r, "g"] \ar[d, "e'"'] \ar[rd, phantom, very near start, "\lrcorner"] & S \ar[d, "e"] \\
        E' \ar[r, "g'"] \ar[d, "p'"'] \ar[rd, phantom, very near start, "\lrcorner"] & E \ar[d, "p"] \\
        S' \ar[r, "g"'] & S.
    \end{tikzcd}
    \] Then the map \(g^\ast e^!(\unit) \to e'^!(\unit)\) is an isomorphism.

    Consequently, any map \(f\colon X \to S\) that is locally on the source decomposed as a composite of a zero section and a topologically smooth morphism also has the isomorphism \(g'^\ast f^!(\unit) \to f'^!(\unit)\).
\end{lemma}
\begin{proof}
    The first claim follows from \cref{prop:PD}. See~\cite{Sch6FF}, for example.
    The second claim follows from the natural duality pairing \( e^!(\unit) \otimes e^\ast p^!(\unit) \simeq \unit \) applied to the symmetric monoidal functor \( g^\ast \).

    We show the final claim. By descent, it suffices to see that the map \(j'^\ast g'^\ast f^!(\unit) \to j'^\ast f'^!(\unit)\) is an isomorphism for \(j\colon U\hookrightarrow X\) open such that \(fj\) decomposes into a zero section \(e\colon U \to E\) and a smooth morphism \(p\colon E \to S\). \[
    \begin{tikzcd}
        U' \ar[r, "g''"] \ar[d, "j'"'] \ar[rd, phantom, very near start, "\lrcorner"] & U \ar[d, "j"] \ar[rd, "e"] & \\
        X' \ar[r, "g'"] \ar[d, "f'"'] \ar[rd, phantom, very near start, "\lrcorner"] & X \ar[d, "f"] & E \ar[ld, "p"] \\
        S' \ar[r, "g"'] & S
    \end{tikzcd}
    \] The map \(j'^\ast g'^\ast f^!(\unit) \to j'^\ast f'^!(\unit)\) is thus identified with the map \[
    g''^\ast \underbrace{e^! p^!(\unit)}_{e^!(\unit)\otimes e^\ast p^!(\unit)} \to \underbrace{e'^!p'^!(\unit)}_{e'^!(\unit)\otimes e'^\ast p'^!(\unit)}
    \] and hence an isomorphism since \(g''^\ast\) is symmetric monoidal.
\end{proof}

\section{Recollection: Banach manifold theory}

\subsection{Vector bundles}

A \emph{Banach space} is a complete normed topological vector space (over \(\RR\)).
Let \(\Ban_\RR\) denote the category of Banach spaces and continuous linear maps between them.

\begin{convention}[Banach/Hilbert bundles]\label{conv:HilbBdl}
    For Banach spaces \(\mathcal{H}\) and \(\mathcal{H}'\), we put the norm topology on the space of continuous/bounded linear maps \(\Hom_{\Ban_\RR}(\mathcal{H},\mathcal{H}')\)%\footnote{It is the internal hom for \(\Ban_\RR^{\le1}\in \CAlg(\Pr^{\L,\aleph_1})\), the presentably symmetric monoidal category of Banach spaces and linear weak contractions.}%
    .
    A pleasant property for this topology is that the evaluation and the composition maps \begin{gather*}\mathcal{H}\times \Hom_{\Ban_\RR}(\mathcal{H},\mathcal{H}') \to \mathcal{H}' \\ \Hom_{\Ban_\RR}(\mathcal{H},\mathcal{H}') \times \Hom_{\Ban_\RR}(\mathcal{H}',\mathcal{H}'') \to \Hom_{\Ban_\RR}(\mathcal{H},\mathcal{H}'')\end{gather*} are continuous.
    The notion of \emph{Banach vector bundles} (over a topological space) will always be interpreted in this topology: such notion behaves well since there exists a topological group structure on the space of continuous linear automorphisms (\cref{prop:GL(H)}).
    We would like to define a \emph{Hilbert vector bundle} to be a Banach vector bundle whose fibers admit a further structure of a Hilbert space.%\footnote{So the term \emph{Hilbertable} might have been more precise.}
    
    We say a (continuous) map between Banach vector bundles a \emph{linear map} if it can be locally trivialized to a continuous family of bounded linear maps (in terms of the norm topology).
    In the same way, we can consider \emph{linear Fredholm maps} between Banach vector bundles.
\end{convention}

The following well-known properties are fundamental.

\begin{proposition}\label{prop:GL(H)}
    Let \(\mathcal{H}\) and \(\mathcal{H}'\) be Banach spaces. Then:
    \begin{enumerate}
    \item
    A continuous linear bijection \(\mathcal{H} \to \mathcal{H}'\) admits a continuous inverse, i.e., it is an isomorphism in the category \(\Ban_\RR\).
    \item
    The subspace \(\Iso_{\Ban_\RR}(\mathcal{H},\mathcal{H}')\) of continuous linear isomorphisms is open in the norm topology.
    \item
    The group \(\Aut_{\Ban_\RR}(\mathcal{H})\) of continuous linear automorphisms is a topological group in the norm topology, in the sense that the inverse operation is also continuous.
    \end{enumerate}
\end{proposition}

The following properties are immediate consequences.

\begin{corollary}\label{cor:submersiveKernel}
    Let \(E \to E'\) be a linear map of Hilbert vector bundles over a base \(S\).
    If it is surjective on fibers, then the kernels form a Hilbert vector bundle.

    Even if we consider a linear map \(\varphi\colon E \to E'\) of Banach vector bundles, the kernels form a Banach vector bundle whenever the map is surjective on fibers and the (fiberwise) kernels \(\ker(\varphi_x)\) have complementary closed linear subspaces in those fibers \(E_x\).
\end{corollary}

\begin{corollary}
    Let \(E \to E'\) be a linear map of Banach vector bundles over a base \(S\).
    If it is injective on fibers and the (fiberwise) images \(\im(\varphi_x)\) have complementary closed linear subspaces in those fibers \(E_x\), then the cokernels form a Banach vector bundle.
\end{corollary}

\begin{corollary}\label{cor:SurjOpen}
    Let \(\varphi\colon E \to E'\) be a linear map between Banach vector bundles over \(S\).
    Assume that each kernel \(\ker(\varphi_x)\) admits a complementary closed linear subspace in \(E_x\).
    Then the subset \(\{x\in S \mid \varphi_x \text{ is a surjection}\}\) is open in \(S\).
\end{corollary}

\subsection{Some calculus on vector bundles}
We provide a variant of the inverse function theorem for Banach vector bundles.

\begin{definition}\label{dfn:bundleC^1}
    Let \(E\) and \(E'\) be (the total spaces of) Banach vector bundles over a base \(S\).
    Let \(f\colon V \to E'\) be a map defined on some open subset \(V\subset E\).
    Then \(f\) is said to be \emph{of \(C^1\)-class over \(S\)} or simply \emph{\(C^1\)-differentiable} if the differential on the vertical directions \[
    df \colon
    V \times_S E \to V \times_S E'
    \] defines a continuous linear map (in the norm topology: \cref{conv:HilbBdl}) over the base \(V\).
\end{definition}

\begin{lemma}
    Let \(U\times \mathcal{H}\) and \(U\times \mathcal{H}'\) be trivial Banach vector bundles over a topological space \(U\).
    Let \(f\) be a map \(U\times \mathcal{H} \to U \times \mathcal{H}'\) of \(C^1\)-class over \(U\) (\cref{dfn:bundleC^1}).
    Assume that at some point \((x_0,v_0)\in U\times \mathcal{H}\), the vertical derivative
    \[d_{v_0}f_{x_0} \colon \mathcal{H} \to \mathcal{H}'\]
    is a (continuous) linear isomorphism.
    Then \(f\) admits a local \(C^1\)-inverse near \((x_0,v_0)\), i.e., there exists a local inverse \[g\colon U_1 \times V' \to U_1\times V\]
    which is of \(C^1\)-class over \(U_1\).\footnote{We don't use the \(C^1\)-structure of the local inverse in this paper.}
\end{lemma}
\begin{proof}
    We argue by a completely analogous argument to~\cite[Theorem 5.2]{Lang}.

    We will write the value of the map \(f\) at a point \((x,v)\in U\times \mathcal{H}\) as \(f_x(v)\). Via the isomorphism \(d_{v_0} f_{x_0}\), we may identify \(\mathcal{H}'\) with \(\mathcal{H}\) and \(d_{v_0} f_{x_0}\) with \(\id_{\mathcal{H}}\). We may also translate and assume that \(v_0=0\).

    First, consider the map \(p_x(v) \coloneqq (x, v-f_x(v)) \colon U\times \mathcal{H} \to U\times \mathcal{H}\), so that \(d_0 p_{x_0}=0\). By the norm-continuity, we find an open neighborhood \(U_0\) of \(x_0\) and a positive \(r>0\) such that \(\|d_v p_x\|<\frac{1}{2}\) holds for \(x\in U_0\) and \(|v|<r\).
    By the mean value theorem~\cite[Lemma 4.2]{Lang}, it follows that \[p_x(v_1) - p_x(v_2) \in \{x\}\times B_{\le r/2}\] for \(v_1,v_2 \in B_{\le r}\) and \(x\in U_0\).
    In other words, \(p\) maps \(U_0 \times B_{\le r}\) to \(U_0 \times B_{\le r/2}\).

    Next, consider the map \(q_{x,w} \colon \mathcal{H} \to \mathcal{H} ; q_{x,w}(v) \coloneqq w+ v- f_x(v)\).
    For each  \(x\in U_0\) and \(|w|\le \frac{r}{2}\), we have
    \[
    | q_{x,w}(v) | \le |w| + |v-f_x(v)| \le \frac{r}{2} + \frac{r}{2}
    \] for any \(v\in B_{\le r}\). So \(q_{x,w}\) is a map from \(B_{\le r}\) to itself.
    Moreover, \(q_{x,w} \colon B_{\le r} \to B_{\le r}\) is a contraction since 
    \[
    | q_{x,w}(v_1) - q_{x,w}(v_2) | \le \frac{1}{2} |v_1 - v_2|
    \] again by the mean value theorem.
    Thus, by the fixed point theorem, for every \(x\in U_0\) and \(|w|\le B_{\le r/2}\) there exists a unique \(v\in B_{\le r}\), which will be denoted by \(g_x(w)\), such that \(w = f_x(v)\).

    We then want to conclude that \(g\colon U_0 \times B_{\le r/2} \to B_{\le r}\) is continuous.
    We have the following inequalities
    \begin{align*}
        |g_{x_1}(w_1) - g_{x_2}(w_2)| & = | p_{x_1}g_{x_1}(w_1) + f_{x_1}g_{x_1}(w_1) - p_{x_1}g_{x_2}(w_2) - f_{x_1}g_{x_2}(w_2) | \\
        &\le | p_{x_1}(g_{x_1}(w_1)) - p_{x_1}(g_{x_2}(w_2) | + | w_1 - f_{x_1}g_{x_2}(w_2) | \\
        &\le \frac{1}{2}|g_{x_1}(w_1) - g_{x_2}(w_2) | + |w_1 - w_2| + | f_{x_2}(g_{x_2}(w_2)) - f_{x_1}(g_{x_2}(w_2)) |
    \end{align*}
    so that 
    \[
    |g_{x_1}(w_1) - g_{x_2}(w_2)| \le 2 |w_1-w_2| + 2|f_{x_2}(v_2) - f_{x_1}(v_2)|
    \]
    and the continuity of \(g\) follows from the continuity of \(f\colon U \times \mathcal{H} \to \mathcal{H}'\).
    This completes the proof that \(f\) is a local homeomorphism at that point.

    We don't prove the continuity of the differential of \(g\) because we don't need this.
\end{proof}

%We therefore have an immediate corollary.

\begin{corollary}[Submersions are smooth] \label{cor:submersion}
    Let \(U\times \mathcal{H}\) and \(U\times \mathcal{H}'\) be trivial Banach vector bundles over a topological space \(U\).
    Let \(f\) be a map \(U\times \mathcal{H} \to U \times \mathcal{H}'\) of \(C^1\)-class over \(U\) (\cref{dfn:bundleC^1}).    
    Assume that at some point \((x,v)\in U\times \mathcal{H}\) the vertical differential
    \[
    d_{v}f_{x} \colon \mathcal{H} \to \mathcal{H}'
    \] is surjective and that \(W_0 \coloneqq \ker(d_{v}f_{x})\) admits a complementary closed linear subspace in \(\mathcal{H}\).
    Then there exist open subsets \(U_1\times V \ni (x,v)\) of \(U\times \mathcal{H}\), \(V'\) of \(\mathcal{H}'\)%, \(W\subset \ker(d_{v_0}f_{x_0})\)
    and a homeomorphism%\footnote{In fact, this can be chosen to be a \(C^1\)-diffeomorphism.} %\(V' \times W \cong V\) such that the diagram 
    \(V' \times W_0 \cong V\) such that the diagram
    \[\begin{tikzcd}[column sep=small]
        U_1\times V'\times W_0 \ar[d, "\pr"] \ar[r, phantom, "\cong"] & U_1\times V \ar[r, phantom, "\subset"] & U\times \mathcal{H} \ar[dl, "f", end anchor=north east] \\
        U_1\times V' \ar[r, phantom, "\subset"]
        & U_1 \times \mathcal{H}'
    \end{tikzcd}\] commutes.
\end{corollary}

\providecommand{\bysame}{\leavevmode\hbox to3em{\hrulefill}\thinspace}
\renewcommand{\MR}[1]{\relax\ifhmode\unskip\space\fi MR: \href{http://www.ams.org/mathscinet-getitem?mr=#1}{#1}.}

\end{document}